\documentclass[letterpaper,11pt  ]{article}

\usepackage[utf8]{inputenc}
\usepackage[T1]{fontenc}
\usepackage{lmodern}
\usepackage[a4paper]{geometry}
\usepackage[frenchb, english]{babel}
\usepackage{amsthm}
\usepackage{amsmath, amssymb}
\usepackage{tikz}
\usepackage{graphicx}
  \usepackage[all]{xy}
\usepackage{hyperref}
  \newcommand{\eq}[1][r]
   {\ar@<-3pt>@{-}[#1]
    \ar@<-1pt>@{}[#1]|<{}="gauche"
    \ar@<+0pt>@{}[#1]|-{}="milieu"
    \ar@<+1pt>@{}[#1]|>{}="droite"
    \ar@/^2pt/@{-}"gauche";"milieu"
    \ar@/_2pt/@{-}"milieu";"droite"}

\newtheorem*{thmq}{Theorem 1 in [Gue16-1]}

\newtheorem{thm}{Theorem}
\newtheorem{lemme}{Lemma}
\newtheorem{prop}{Proposition}
\newtheorem{ex}{Example}
\newtheorem{cor}{Corollary}
\newtheorem{thmr}{Result}

\DeclareMathOperator{\Tr}{Tr}

\begin{document}
\title{Centralizers of irreducible subgroups in the projective special linear group}

\author{Cl\'ement Gu\'erin\thanks{
              University of Luxembourg Campus Kirchberg, Mathematics Research Unit, BLG. 6, rue Richard Coudenhove-Kalergi, L-1359 Luxembourg,
            e-mail : clement.guerin@uni.lu} }
\maketitle

\begin{abstract}
In this paper, we classify  conjugacy classes of centralizers of irreducible subgroups in $PSL(n,\mathbb{C})$ using alternate modules a.k.a. finite abelian groups with an alternate bilinear form. When $n$ is squarefree, we prove that these conjugacy classes are classified by their isomorphism classes. More generally, we define a finite graph related to this classification whose combinatorial properties are expected to help us describe the stratification of the singular (orbifold) locus in some character varieties. 

\end{abstract}

\tableofcontents

\section{Definitions and results}\label{sec1}

A bad  subgroup of a complex reductive Lie group $G$ is an irreducible subgroup $H$ of $G$ (see the definition below) whose centralizer strictly contains the center of $G$. Sikora in \cite{Sik} and Florentino-Lawton in \cite{F-L} exhibited complex reductive Lie groups $G$ containing bad  subgroups. Remark that applying Schur's lemma, there are no bad  subgroups in $SL(n,\mathbb{C})$ nor in $GL(n,\mathbb{C})$. A first case (when $G$ is $PSL(p,\mathbb{C})$ and $p$ is a prime number)  has been extensively studied in \cite{Gue} from which we highlight :
 
\begin{thmq}\label{thmgue}
If $\overline{H}$ is a bad subgroup of $PSL(p,\mathbb{C})$ then its centralizer $Z_{PSL(p,\mathbb{C})}(\overline{H})$ is either isomorphic to $\mathbb{Z}/p$ or $\mathbb{Z}/p\times \mathbb{Z}/p$, in the later case the irreducible subgroup is ts own centralizer. Furthermore if two centralizers of irreducible subgroups of $PSL(p,\mathbb{C})$ are isomorphic then they are conjugate. 
\end{thmq}

This result  leads to a  decomposition of the singular locus of the character variety  (see loc. cit. for a definition and also paragraph 7 of \cite{F-L-R}) for Fuchsian groups into $PSL(p,\mathbb{C})$. In this paper, we generalize the classification of centralizers of irreducible subgroups of $PSL(p,\mathbb{C})$ to $PSL(n,\mathbb{C})$. We recall some definitions for  this paper.

\bigskip

A subgroup $P$ of a reductive group $G$ is said to be $\textit{parabolic}$ if $G/P$ is a complete variety. When  $G$ is $SL(n,\mathbb{C})$ a subgroup is parabolic if and only if it is the stabilizer of a non-trivial flag in  $\mathbb{C}^n$ where $SL(n,\mathbb{C})$ acts canonically on $\mathbb{C}^n$  (c.f. \cite{Bor}).

A subgroup $H$ of a reductive group $G$ is said to be \textit{irreducible} if no parabolic subgroup of $G$ contains $H$. A subgroup $H$ of a reductive group $G$ is said to be \textit{completely reducible} if for each parabolic subgroup $P$ of $G$ containing $H$, we can find  a Levi subgroup $L$ of $P$ such that $H\subseteq L$. A representation $\rho:\Gamma\rightarrow G$ is said to be \textit{irreducible} (resp. \textit{completely reducible}) if $\rho(\Gamma)$ is irreducible (resp. completely reducible).

The \textit{centralizer}  $Z_G(H)$ of a subgroup $H$ of $G$ is the set of elements $g\in G$ commuting with any element of $H$. The \textit{centralizer} $Z_G(\rho)$ of a representation $\rho: \Gamma\rightarrow G$ is the centralizer of $\rho(\Gamma)$.

Sikora gave   a useful characterization of an irreducible group (corollary 17 in  \cite{Sik}). A group $H$ in a reductive group $G$ is irreducible if and only if it is completely reducible and $[Z_G(H):Z(G)]$ is finite. 

Furthermore any finite group is a completely reducible subgroup and finite extensions of completely reducible subgroups are completely reducible subgroups.

The \textit{commutator} $[g,h]$ of $g$ and $h$ in a group $G$ will classically be defined as $ghg^{-1}h^{-1}$.  We recall a simple lemma from a previous paper :
 
\begin{lemme}\label{2.2l1}

Let $n\geq 1$, $A,B$ be two matrices in $GL(n,\mathbb{C})$ and $\lambda\in \mathbb{C}^*$ such that their commutator verifies  $[A,B]=\lambda I_n$, then for all $\mu\in \mathbb{C}^*$, $A(E_{\mu}(B))=E_{\lambda^{-1}\mu}(B)$ and $B(E_{\mu}(A))=E_{\lambda \mu}(A)$. In particular, $B$ acts on $Sp(A)$ by multiplying by $\lambda$ and $A$ acts on $Sp(B)$ by multiplying by $\lambda^{-1}$.

\end{lemme}

For $n\geq 1$, the center of $SL(n,\mathbb{C})$ is cyclic of order $n$. For $d$ dividing $n$,  denote $\pi_d(SL(n,\mathbb{C}))$ the quotient of $SL(n,\mathbb{C})$ by the unique central subgroup of order $d$. This gives all the quotients of $SL(n,\mathbb{C})$, in particular $\pi_n(SL(n,\mathbb{C}))=PSL(n,\mathbb{C})$.

\bigskip

The first result of this paper (see proposition \ref{2p1} and corollary \ref{3.ncardinal}) in subsection \ref{sec2sub1} :

\begin{thmr}\label{r1}

Let $n\geq 1$ and $\overline{H}$ be an irreducible subgroup of $PSL(n,\mathbb{C})$, then $Z_{PSL(n,\mathbb{C})}(\overline{H})$ is abelian, of exponent dividing $n$ and of order dividing $n^2$. 

\end{thmr}

In subsections \ref{sec2sub2} and \ref{sec3sub1}, we classify conjugacy classes of centralizers of irreducible subgroups in $PSL(n,\mathbb{C})$ using alternate modules which are, by definition, abelian finite groups endowed with an alternate bilinear form, see \cite{T-W} or \cite{Wal}. We show, in proposition \ref{3.assoc}, that we can associate to any centralizer of an irreducible subgroup in $PSL(n,\mathbb{C})$, a unique isometry class of alternate modules. According to this association, we get theorem \ref{classif1}   :

\begin{thmr}\label{r2}

Let $n$ be a positive integer and $Z_1$, $Z_2$ be two centralizers of irreducible subgroups in $PSL(n,\mathbb{C})$, then $Z_1$ and $Z_2$ are conjugate if and only if their respective associated alternate modules are isometric. 

\end{thmr}

Then, we give a necessary and sufficient condition for an alternate module to be associated to a centralizer of irreducible subgroup in $PSL(n,\mathbb{C})$ in theorem \ref{classif2} :

\begin{thmr}\label{r3}

Let $n\geq 1$ and $(A,\phi)$ be an alternate module then the following assertions are equivalent :

1. There exists an irreducible subgroup of $PSL(n,\mathbb{C})$ such that the alternate module associated to its centralizer is isometric to $(A,\phi)$.

2. The order of Lagrangians in $(A,\phi)$ divides $n$.

3. There exists an abelian group $B$ of order $n$ such that $(A,\phi)$ is isometrically embedded in the symplectic module $B\times B^*$.

\end{thmr}

This theorem is the generalization of theorem 1 in \cite{Gue}. Alternate modules verifying the assertion 3 will be referred to as $n$-\textit{subsymplectic} modules. When $n=p$ is prime,   conjugacy classes of centralizers are classified by their isomorphism class. In the general case, we prove theorem \ref{sqf} and corollary \ref{sqfnumber} in subsection~\ref{sec3sub2}~:

\begin{thmr}\label{r4}
  Conjugacy classes of centralizers of irreducible subgroups in $PSL(n,\mathbb{C})$ are classified by their isomorphism classes if and only if $n$ is squarefree. When $n$ is squarefree,  there are exactly $3^r$ conjugacy classes of centralizers of irreducible subgroups in $PSL(n,\mathbb{C})$ where $r$ is the number of distinct prime numbers dividing $n$.

\end{thmr}

After this, we define $(\mathcal{M}_n,\leq)$ to be the set of isometry classes of $n$-subsymplectic modules ordered with the inclusion up to isometry. This allows us to draw a graph $G_n$ ruling those inclusions (see  examples \ref{Mp1}, \ref{Mp2} and~\ref{Mp3}). When $\Gamma$ is a finitely generated group,   we show in proposition \ref{duality}, that, by duality,  this gives a stratification of the singular locus of the character variety $\chi_{Sing}^i(\Gamma,PSL(n,\mathbb{C}))$.  In particular, one should be able to prove similar results to those that were proven in \cite{Gue} when $\Gamma$ is Fuchsian.

\bigskip

Finally, in section \ref{sec4}, we deal with similar questions in partial quotients of $SL(n,\mathbb{C})$. We first get theorem \ref{casd} which generalizes result \ref{r1} :

\begin{thmr}\label{r5}

Let $n$ be a positive integer, $d$ be a divisor of $n$ and $\overline{H}$ be an irreducible subgroup of $\pi_d(SL(n,\mathbb{C}))$. Then its centralizer in $\pi_d(SL(n,\mathbb{C}))$ is abelian of exponent dividing $lcm(n/d,d)$ and of order dividing $n^3/d$. If $gcd(n/d,d)=1$ then the order of its centralizer in $\pi_d(SL(n,\mathbb{C}))$ divides  $nd$. 

\end{thmr}

Likewise, propositions \ref{dclassif1} and \ref{dclassif2} respectively generalize  results \ref{r2} and \ref{r3}. From section~\ref{sec4}, we deduce that classifying conjugacy classes of centralizers in $PSL(n,\mathbb{C})$ is equivalent to classifying conjugacy classes of centralizers in all quotients of $SL(n,\mathbb{C})$. In corollary \ref{isocarvar}, we characterize the isotropy group of the corresponding character variety when it is an orbifold (e.g. when $\Gamma$ is Fuchsian, see \cite{Sik}).

\section{Properties of centralizers of irreducible in $PSL(n,\mathbb{C})$}\label{sec2}

In this section, we will demonstrate that any centralizer of an irreducible subgroup in $PSL(n,\mathbb{C})$ is abelian, of bounded exponent and  of bounded order (first subsection). In the second subsection, we shall see how to associate an alternate module to any centralizer of an irreducible subgroup in $PSL(n,\mathbb{C})$. The correspondence will be proven to be faithful.

\begin{lemme}\label{2l1}

Let $G$ be a group and $N$ a subgroup of $Z(G)$, the center of $G$. Define $\pi:G\rightarrow G/N$ the quotient map.   If $\overline{H}$ is a subgroup of $G/N$, define $H:=\pi^{-1}(\overline{H})$ and  $  U:=\pi^{-1}(Z_{G/N}(\overline{H})) $. The function

\begin{displaymath}
\phi :
\left|
  \begin{array}{rcl}
    U& \longrightarrow &Mor(H,N) \\
   u& \longmapsto & (h\mapsto [u,h]) \\
  \end{array}
\right.
\end{displaymath}

is well defined. It is a group morphism whose kernel is $Z_G(H)$.

\end{lemme}

\begin{proof}

Let $u\in U$ and $h\in H$ then $\pi([u,h])=[\pi(u),\pi(h)]=N$ since $\pi(u)$ centralizes $\pi(h)$ by definition. Furthermore, if $h_1,h_2\in H$ then  :

$$[u,h_1h_2]=uh_1h_2u^{-1}h_2^{-1}h_1^{-1}=uh_1(u^{-1}u)h_2u^{-1}h_2^{-1}h_1^{-1}=uh_1u^{-1}[u,h_2]h_1^{-1}$$

Since $[u,h_2]$ is central, $[u,h_1h_2]=uh_1u^{-1}h_1^{-1}[u,h_2]=[u,h_1][u,h_2]$. As a result, $\phi_u:h\mapsto [u,h]$ is a morphism from $H$ to $N$ and $\phi$ is well defined.  Using the exact same kind of argument,  $\phi$ is easily proven to be itself a group morphism.

An element $u$ in  $U$ belongs to  $Ker(\phi)$ if and only if for all $h\in H$ we have  $[u,h]=1_G $, if and only if $u\in Z_G(H)$.  \end{proof}

 Once $n$ is given, $\xi$ will always denote a fixed primitive $n$-th root of the unity in the complex field. Let $n$ be a positive integer and $d$ a divisor of $n$. We define the natural projection  $\pi_d:SL(n,\mathbb{C})\rightarrow SL(n,\mathbb{C})/\langle\xi^{\frac{n}{d}}I_n\rangle$.

\bigskip

Quotients of $SL(n,\mathbb{C})$ are of the form $\pi_d(SL(n,\mathbb{C}))$, where $d$ divides $n$. We will study centralizers of their irreducible subgroups. Let $H$ be a subgroup of $SL(n,\mathbb{C})$, we define its $d$-\textit{centralizer} $Z_d(H):=Z_{\pi_d(SL(n,\mathbb{C}))}(\pi_d(H))\leq \pi_d(SL(n,\mathbb{C}))$.

\bigskip

Working in $SL(n,\mathbb{C})$ rather than in its quotients is a natural thing to do. Therefore, we define for $d$ dividing $n$ and a subgroup $H$  of $SL(n,\mathbb{C})$, the $d$-\textit{extended centralizer} $U_d(H)$ of $H$ as the pull-back of $Z_d(H)$ by $\pi_d$. From the very definition of $U_d(H)$ :

$$U_d(H)=\{g\in SL(n,\mathbb{C})\mid \forall h\in H, \exists k\in \mathbb{Z}/d\text{ such that } [g,h]=\xi^{\frac{n}{d}k}I_n\}. $$

\bigskip

 Remark that if $\overline{H}$ is an irreducible subgroup of $\pi_d(SL(n,\mathbb{C}))$ then $H:=\pi_d^{-1}(\overline{H})$ is an irreducible subgroup of $SL(n,\mathbb{C})$. Furthermore $\pi_d(\pi_d^{-1}(\overline{H}))=\overline{H}$  and its centralizer  $Z_{\pi_d(SL(n,\mathbb{C}))}(\overline{H})$ is equal to $Z_d(H)$. As a result,  rather than studying  centralizers of  irreducible subgroups of quotients of $SL(n,\mathbb{C})$, we can equivalently study   $d$-extended centralizers of irreducible subgroups of $SL(n,\mathbb{C})$ for $d$ dividing $n$.

\subsection{Abelianity, exponent and order}\label{sec2sub1}

\begin{prop}\label{2p1}

Let $n\geq 1$ and $H$ be an irreducible subgroup of $SL(n,\mathbb{C})$. Then, the $n$-centralizer $Z_n(H)$ of $H$ is abelian of exponent dividing $n$.

\end{prop}

\begin{proof}

According to lemma \ref{2l1}, define the group morphism :

\begin{displaymath}
\phi_{n}:
\left|
  \begin{array}{rcl}
    U_n(H)& \longrightarrow &Mor(H,\langle\xi I_n\rangle) \\
   u& \longmapsto & (h\mapsto [u,h]) \\
  \end{array}
\right. 
\end{displaymath}

The kernel of $\phi_n$ is the centralizer $Z_{SL(n,\mathbb{C})}(H)$ of $H$ in $SL(n,\mathbb{C})$. Since $H$ is irreducible, Schur's lemma implies that  $Z_{SL(n,\mathbb{C})}(H)=Z(SL(n,\mathbb{C}))=\langle\xi I_n\rangle$. Hence, the group $U_n(H)/Ker(\phi_n)=U_n(H)/\langle\xi I_n\rangle=\pi_n^{-1}(Z_n(H))/\langle\xi I_n\rangle$ is equal to $Z_n(H)$. On the other hand : $U_n(H)/Ker(\phi_n)$  is isomorphic to a subgroup of $Mor(H,\langle\xi I_n\rangle)$.

\bigskip

As a result, $Z_n(H)$ is isomorphic to a subgroup of $Mor(H,\langle\xi I_n\rangle)$. Since this group is abelian of exponent $n$, it follows that $Z_n(H)$ is also abelian of exponent dividing $n$.\end{proof}

The next proposition  justifies that the conjugacy class of any element in the $n$-extended centralizer of an irreducible subgroup in $SL(n,\mathbb{C})$ is well understood.

\begin{prop}\label{3.diagonal}

Let $n\geq 1$ and  $H$ be an irreducible subgroup of $SL(n,\mathbb{C})$. If $u\in U_n(H)$ and $\pi_n(u)$ is of order $d$ in $Z_n(H)$ then there exists  $\lambda\in \mathbb{C}^*$ such that :

$$u\text{ is conjugate to }\lambda \begin{pmatrix}I_{\frac{n}{d}}&&&\\&\xi^{\frac{n}{d}}I_{\frac{n}{d}}&&\\&&\ddots&\\&&&\xi^{\frac{n}{d}(d-1)}I_{\frac{n}{d}}\end{pmatrix}\text{,} $$

$$\text{ where } \lambda\in \left\lbrace
\begin{array}{ll}
\langle \xi I_n\rangle  &  \mbox{if $d$ is odd or $d$ even and $n/d$ even,}\\
\sqrt{\xi^{-\frac{n(d-1)}{d}}}\langle \xi I_n\rangle& \mbox{if $d$ is even and $n/d$ odd}\text{.} 
\end{array}
\right.$$

\end{prop}

\begin{proof}

First, $\pi_n(H)$ is irreducible (since $H$ is) so $Z_n(H)$ is finite and $U_n(H)$ is also finite. It follows that $u$ is necessarily of finite order and, in particular, it is diagonalizable.

\bigskip

For all $h\in H$, there exists $s_h\in\mathbb{Z}/n$ such that  $[h,u]=\xi^{s_h}I_n $. Applying lemma \ref{2l1}, the application $s:\left|\begin{array}{rcl}
H&\longrightarrow & \mathbb{Z}/n\\h&\longmapsto& s_h\end{array}\right.$ is a group morphism. Let $\xi^t$ be a generator of $s(H)$. 

\bigskip

Lemma \ref{2.2l1} implies that $H$ acts on the spectrum  $Sp(u)$ of $u$ (i.e. the set of its eigenvalues) :  $\text{ if } h\in H\text{ and } \mu \in Sp(u)$ then $h\cdot \mu:=\xi^{s_h}\mu $. Let $X$ be an orbit in  $Sp(u)$ for this action $H$.  The  subspace $\underset{\mu\in X}{\bigoplus}E_{\mu}(u)$ of  $\mathbb{C}^{\mathbb{Z}/n}$ is stable by  $H$.

Since  $H$ is irreducible, this non-trivial subspace is the whole space $\mathbb{C}^{\mathbb{Z}/n}$. In particular,  the action of  $H$ on   $Sp(u)$ is transitive  and all the eigenspaces have the same dimension $v>0$. Remarking that $H$ acts through the group morphism   $s$ whose image is generated by  $\xi^t$, we can say (if $\lambda$ is some eigenvalue of $u$)  that  $Sp(u)=\lambda\langle \xi^t\rangle$  and since  $u$ is diagonalizable :

$$u\text{ is conjugate to   }\lambda \begin{pmatrix}I_{v}&&&\\&\xi^{t}I_v&&\\&&\ddots&\\&&&\xi^{t(\frac{n}{t}-1)}I_v\end{pmatrix}\text{.}  $$

Using the dimensions, $n=|Sp(u)|\times v=\frac{n}{t}v$, hence $t=v$. Since   $\pi_n(u)$ is of order $d$, this implies that $t=\frac{n}{d}k$ with $0<k<d$ prime to  $d$. Hence :

$$u\text{ is conjugate to }\lambda \begin{pmatrix}I_{\frac{n}{d}}&&&\\&\xi^{\frac{n}{d}}I_{\frac{n}{d}}&&\\&&\ddots&\\&&&\xi^{\frac{n}{d}(d-1)}I_{\frac{n}{d}}\end{pmatrix} \text{.}  $$

The condition on  $\lambda$ is given by writing $\det(u)=1$. Which leads to $\lambda^n\xi^{\frac{n}{d}\frac{n(d-1)}{2}}=1$.

\bigskip

If $d$ is odd, then   $2$ divides $d-1$ and then  $\lambda^n=1 $.

If $d$ is even and  $2$ divides $\frac{n}{d}$, then   $\lambda^n=\xi^{-\frac{n}{d}\frac{n(d-1)}{2}}=(\xi^n)^{\frac{n}{2d}(d-1)}=1$.

If $d$ is even and $2$ does not divide $\frac{n}{d}$, then  $\lambda^n=\xi^{-\frac{n}{d}\frac{n(d-1)}{2}}$ whence :

 $$\lambda\equiv\xi^{-\frac{n(d-1)}{2d}}\equiv\sqrt{\xi^{-\frac{n(d-1)}{d}}}\text{ mod } \langle \xi I_n\rangle\text{.} $$

\end{proof}

Let $n\geq 1$ and $H$ be an irreducible subgroup of $SL(n,\mathbb{C})$, we define the  \textit{standard representation} of $U_n(H)$ as the natural inclusion $\iota_H$ of $U_n(H)$ in $SL(n,\mathbb{C})$. Its character $\Tr\circ \iota_H$ will be denoted $\chi_H$. Computing this character appears to be easy.

\begin{prop}\label{3.car}
Let $n\geq 1$, $H$ be an irreducible subgroup in $SL(n,\mathbb{C})$ and $u$ in $U_n(H)$,

$$\text{then : } \chi_H(u)=\left\lbrace
\begin{array}{ll}
0  &  \mbox{if $u\notin \langle\xi I_n\rangle$}\\
n\xi^k&\mbox{if $u=\xi^kI_n$ where $k\in\mathbb{Z}/n$}\text{.} 
\end{array}
\right.$$

\end{prop}
 
\begin{proof}

Let  $u\in U_n(H)$, if $u=\xi^kI_n\in \langle \xi I_n\rangle$ then, $\chi_H(u)=\Tr(\iota_H(u))=\Tr(\xi^kI_n)=n\xi^k$.

If  $u$ is not central,  then  $u$ is not trivial in   $U_n(H)/\langle \xi I_n\rangle=Z_n(H)$. Let $d>1$ be its order in $Z_n(H)$. By  proposition \ref{3.diagonal} which explicitely gives the conjugacy class of $u$, there exists  $\lambda\in \mathbb{C}^*$ such that   $\chi_H(u)=\lambda\frac{n}{d}(1+\xi^{\frac{n}{d}}+\cdots+\xi^{\frac{n}{d}(d-1)})$ and since $\xi^{\frac{n}{d}}$ is a primitive   $d$-th droot of the unity for $d>1$, the sum $1+\xi^{\frac{n}{d}}+\cdots+\xi^{\frac{n}{d}(d-1)}=0$ , therefore $\chi_H(u)=0$. \end{proof}

Since $U_n(H)$ is a finite group,  we may use the theory of finite group representations in order to have some additional properties on $Z_n(H)$. For instance :

\begin{cor}\label{3.ncardinal}

Let $n\geq 1$ and $H$ be an irreducible subgroup of $SL(n,\mathbb{C})$ then $|Z_n(H)|$ divides $n^2$. Furthermore, $Z_n(H)\text{ is irreducible if and only if } |Z_n(H)|=n^2$.

\end{cor}

\begin{proof}

First, we sum up classical results of the theory of finite groups representations, they can be found in \cite{Ser}. Say we are given a representation $\rho:G\rightarrow GL(n,\mathbb{C})$ of a finite group $G$, then its character $\chi:=\Tr\circ\rho$ has a norm defined by :

$$\|\chi\|^2:=\frac{1}{|G|}\sum_{g\in G}\chi(g)\chi(g^{-1})\text{.} $$

We know that $\|\chi\|^2$ is a natural number and  $\|\chi\|^2=1$ if and only if the representation  $\rho$ is irreducible. Applying this to the standard representation  $\iota_H$ of $U_n(H)$ :

\begin{align*}
\|\chi_H\|^2&=\frac{1}{|U_n(H)|}\sum_{u\in U_n(H)}\chi_H(u)\chi_H(u^{-1}) \underbrace{=}_{\text{ by proposition \ref{3.car}}}\frac{n^3}{|U_n(H)|}\text{.}
\end{align*}

Since $\|\chi_H\|^2$ must be an integer, $|U_n(H)|$ divides $n^3$ and since $|U_n(H)|$ is equal to $n|Z_n(H)|$,  $|Z_n(H)|$ divides $n^2$.  Furthermore, $Z_n(H)$ is irreducible if and only if $\pi_n^{-1}(Z_n(H))=U_n(H)$ is irreducible if and only if $\|\chi_H\|^2=1$   if and only if $|Z_n(H)|=~n^2$.\end{proof}

By the following example, the bound is always reached :

\begin{ex}\label{cardncarre}
Let $n\geq 1$, define :

$$u:=\lambda\begin{pmatrix}1&&&\\&\xi&&\\&&\ddots&\\&&&\xi^{n-1}\end{pmatrix} \text{ and } M:=\lambda\begin{pmatrix}0&&&1\\1&\ddots&&\\&\ddots&\ddots&\\&&1&0\end{pmatrix}$$

where $\lambda=1$ if $n$ is odd and $\sqrt{\xi}$ if $n$ is even. Then $H:=\langle u,M\rangle$ is an irreducible subgroup of $SL(n,\mathbb{C})$ whose $n$-centralizer is $\pi_n(H)$ which is of order $n^2$.

\end{ex}

\begin{proof} See the proof of proposition \ref{ncompatible}. \end{proof}

The second corollary deals with conjugacy classes of centralizers of irreducible subgroups in $PSL(n,\mathbb{C})$. 

\begin{cor}\label{conjugclass}

Let $H_1$ and $H_2$ be two irreducible subgroups in $SL(n,\mathbb{C})$. Then $Z_n(H_1)$ is conjugate to $Z_n(H_2)$ if and only if there exists an abstract group isomorphism $f$ from $U_n(H_1)$ to  $U_n(H_2)$ such that for all $\xi^kI_n\in Z(SL(n,\mathbb{C}))$, $f(\xi^k I_n)=\xi^k I_n$.

\end{cor}

\begin{proof}

We begin with the assumption that  $Z_n(H_1)$ is conjugate to $Z_n(H_2)$ in $PSL(n,\mathbb{C})$. Then there exists $g\in SL(n,\mathbb{C})$ such that $\pi(g)Z_n(H_1)\pi(g)^{-1}=Z_n(H_2)$ and therefore, the element $g$ conjugates $U_n(H_1)$ to $U_n(H_2)$. As a result, the conjugation morphism by $g$ restricted to $U_n(H_1)$ will do the job.

Conversly, say such $f:U_n(H_1)\rightarrow U_n(H_2)$ exists. Recall that for $i=1,2$, $\iota_{H_i}$ defines the injection of $U_n(H_i)$ in $SL(n,\mathbb{C})$. We define the representation $\rho:=\iota_{H_2}\circ f$ from $U_n(H_1)$ into $SL(n,\mathbb{C})$ and $\chi:=tr\circ \rho$ its character. Remark that :

$$\text{ for } u\in U_n(H_1)\text{, } f(u)\left\lbrace \begin{array}{l}\notin Z(SL(n,\mathbb{C}))\text{ if }u\notin Z(SL(n,\mathbb{C}))\\
=\xi^kI_n\text{ if } u=\xi^k I_n\text{.} \end{array} \right. $$

Applying proposition \ref{3.car} to $\chi_{H_1}$ and $\chi_{H_2}$, we have that for $u=\xi^k I_n\in Z(SL(n,\mathbb{C}))$, $\chi_{H_1}(u)=n\xi^k=\chi(u)$ and for $u\notin Z(SL(n,\mathbb{C}))$, $\chi_{H_1}(u)=0=\chi(u)$. In particular, the representations $\iota_{H_1}$ and $\rho$ of $U_n(H_1)$ share the same character. As a result (see \cite{Ser}), they are conjugate. Since the respective images of $\iota_{H_1}$ and $\rho$  are $U_n(H_1)$ and  $U_n(H_2)$, the groups $U_n(H_1)$ and $U_n(H_2)$ are conjugate.  \end{proof}

Last consequence of proposition \ref{3.car} : any subgroup of $Z_n(H)$ remains  the  centralizer of an irreducible subgroup of $PSL(n,\mathbb{C})$.

\begin{cor}\label{3.stabparsgrp}

Let $n\geq 1$ and $H$  be an irreducible subgroup of  $SL(n,\mathbb{C})$.  For any subgroup  $A$  of $Z_n(H)$, $Z_{PSL(n,\mathbb{C})}(Z_{PSL(n,\mathbb{C})}(A))=A$. In particular, any subgroup of $Z_n(H)$ is itself the centralizer of an irreducible subgroup of $PSL(n,\mathbb{C})$.

\end{cor}

\begin{proof}

We denote $B:=Z_{PSL(n,\mathbb{C})}(Z_{PSL(n,\mathbb{C})}(A))$. The inclusion $A\leq B$ is obvious. Since  $\pi_n(H)\leq Z_{PSL(n,\mathbb{C})}(A)$, we deduce  that $ B\leq Z_n(H)$. Hence  $A\leq B\leq Z_n(H)$. 

\bigskip

 Let us first show that $Z_{PSL(n,\mathbb{C})}(B)=Z_{PSL(n,\mathbb{C})}(A)$. Since  $A\leq B$, it is clear that $Z_{PSL(n,\mathbb{C})}(B)\leq Z_{PSL(n,\mathbb{C})}(A)$. Furthermore, if $z$ commutes with $A$ then any element  $b\in B:=Z_{PSL(n,\mathbb{C})}(Z_{PSL(n,\mathbb{C})}(A))$ will commute with $z$ by definition. 

\begin{equation}
\text{Hence $z\in Z_{PSL(n,\mathbb{C})}(B)$, so that  : }Z_{PSL(n,\mathbb{C})}(B)=Z_{PSL(n,\mathbb{C})}(A)\text{.}
\label{3.commutateurs}
\end{equation}

Let $A_0$ (resp. $B_0$) be $\pi_n^{-1}(A)$ (resp. $\pi_n^{-1}(B)$) then  $A_0\leq B_0$, we denote $Z_1(A_0)$ (resp. $Z_1(B_0)$) the centralizer of  $A_0$ (resp. $B_0$) in  $SL(n,\mathbb{C})$. It follows that $Z_1(B_0)\leq Z_1(A_0)$. Let us show that the index of $Z_1(B_0)$ in $Z_1(A_0)$ is finite. Applying lemma  \ref{2l1}, both indices $[U_n(A_0):Z_1(A_0)]$ and $[U_n(B_0):Z_1(B_0)]$ are  finite.  Since :

\begin{align*}
U_n(A_0)&=\pi_n^{-1}(Z_{PSL(n,\mathbb{C})}(A))\text{ by definition}\\
&=\pi_n^{-1}(Z_{PSL(n,\mathbb{C})}(B))\text{ by equation \ref{3.commutateurs}}\\
&=U_n(B_0)\text{ by definition,}
\end{align*}

we get the following equality :

\begin{align*}
[U_n(B_0):Z_1(B_0)]&=[U_n(A_0):Z_1(B_0)]\\
&=[U_n(A_0):Z_1(A_0)][Z_1(A_0):Z_1(B_0)]\\
&\geq  [Z_1(A_0):Z_1(B_0)]
\end{align*}

\begin{equation}
\text{ and  end up with }[Z_1(A_0):Z_1(B_0)]<\infty\text{.}
\label{3.fini}
\end{equation}

We make a proof by contradiction, say $A\neq B$, then $A_0\neq B_0$. Let us show that this contradicts the inequality \ref{3.fini}. We denote  $\rho_{A_0}$ (resp. $\rho_{B_0}$)  the inclusion of $A_0$ (resp. $B_0$) in $SL(n,\mathbb{C})$ whose character is $\chi_{A_0}$ (resp. $\chi_{B_0}$).  Since $A_0< B_0\leq \pi_n^{-1}(Z_n(H))=U_n(H)$, those representations are restrictions of the standard representation $\iota_H$ of $U_n(H)$ whose character has been computed in proposition~\ref{3.car}.

$$\text{ This leads to }\|\chi_{A_0}\|^2=\frac{n^3}{|A_0|}\text{ and }\|\chi_{B_0}\|^2=\frac{n^3}{|B_0|}\text{.}$$

Since  $|A_0|<|B_0|$, we have $\|\chi_{A_0}\|>\|\chi_{B_0}\|$. In terms of finite groups representations (cf. \cite{Ser}), this means that there exists an irreducible  $B_0$-module of $V:=\mathbb{C}^{\mathbb{Z}/n}$ which is decomposed as a non-trivial sum of sub $A_0$-modules. In particular the centralizer of the representation $\rho_{B_0}$ is of infinite index in the centralizer of the representation $\rho_{A_0}$. This contradicts the assertion \ref{3.fini}. As a result,  $A=Z_{PSL(n,\mathbb{C})}(Z_{PSL(n,\mathbb{C})}(A))$.

In order to prove the last assertion of the corollary, it suffices to remark that $Z_{PSL(n,\mathbb{C})}(A)$ contains $\pi_n(H)$ and is, therefore, irreducible.  \end{proof}

In the next subsection, we introduce a correspondence between $n$-centralizer of an irreducible subgroup and alternate modules.

\subsection{The alternate module associated to a $n$-centralizer}\label{sec2sub2}

Definitions and propositions concerning alternate modules used here can be found in \cite{T-W}, \cite{Wal} and \cite{Gue2}. We recall that an alternate module $(A,\phi)$ is an abelian group equipped with a bilinear map $\phi:A\times A\rightarrow \mathbb{Q}/\mathbb{Z}$. We also remark that the group $ \mathbb{Q}/\mathbb{Z}$ contains a unique copy of $\mathbb{Z}/n$ (namely,  the subgroup generated by $\frac{1}{n}$). In the next proposition, we construct an alternate module out of a finite group $G$ containing a central cyclic subgroup $C$ such that $G/C$ is abelian.

\begin{prop}\label{3.assoc}

Let $n$ be a positive integer, $G$ be a finite group containing a central cyclic subgroup $C$ of order $n$ generated by $c_0$. Let $A:=G/C$ and $\pi$ be the natural projection of $G$ onto $A$, assume that $A$ is abelian. Taking for any $a\in A$ an element $\widehat{a}\in G$ such that $\pi(\widehat{a})=a$ (i.e. an arbitrary lift for $a$), we define :

\begin{displaymath}
\phi_G:
\left|
  \begin{array}{rcl}
   A\times A& \longrightarrow &\mathbb{Z}/n\\
 (a,b)& \longmapsto &\text{$\phi_G(a,b)$ verifying } [\widehat{a},\widehat{b}]=c_0^{\phi_G(a,b)}\\
  \end{array}
\right.
\end{displaymath}

Then  $\phi_G$ is well defined (i.e. does not depend on the chosen lifts). Furthermore, $(A,\phi_G)$ is an alternate module whose kernel $K_{\phi_G}$ is the image of the center of $G$ by  $\pi$.

\bigskip

Assume that $G_1$ and $G_2$ are  two groups containing one fixed central copy of $C$. Assume furthermore that $G_i/C=A_i$ is abelian for $i=1,2$. Finally, assume that there exists an isomorphism $f$ from $G_1$ to $G_2$  fixing $C$ point by point,  then the alternate modules $(A_1,\phi_{G_1})$ and $(A_2,\phi_{G_2})$ are isometric.

\end{prop}

\begin{proof}

The fact that $\phi_G$ has its values in $\mathbb{Z}/n$ is clear since $A$ is commutative. The fact that $\phi_G$ does not depend on the chosen lifts $\hat{\cdot}$ is obvious because two different lifts of the same element are the same element up to a central element. Changing one for another does not change anything in the commutator $[\cdot,\cdot]$. Remark furthermore that if $c_0^k$ is given, then $k$ is well defined modulo $n$ since $c_0$ is of order $n$. The fact that $\phi_G$ is bilinear is a consequence of lemma \ref{2l1} and of the abelianity of  $A$.  If $a\in A$ then $[\widehat{a},\widehat{a}]=I_n$. Whence $\phi_G(a,a)=0$ modulo $n$. Therefore $(A,\phi_G)$ is an alternate module.

\bigskip

Let $a\in A$ then $a\in K_{\phi_G}$ if and only if  $ [\widehat{a},\widehat{b}]=I_n$ for all $b\in A$, if and only if $\widehat{a}\in Z(G)$. Hence the radical of $(A,\phi_G)$ is the projection in $A$ of the center  of $G$.

\bigskip

Let $G_1$ and $G_2$ be like in the proposition. Then $f$ factors through $C\leq G_1$ and $f(C)=C\leq G_2$ and induces a group isomorphism $\overline{f}$ between  $A_1$ and $A_2$. If $a,b\in A_1$ :

\begin{align*}
\xi^{\phi_{G_2}(\overline{f}(a),\overline{f}(b))}&=[\widehat{\overline{f}(a)},\widehat{\overline{f}(b)}]\text{ by definition of  $\phi_{G_2}$}\\
&=f([\widehat{a},\widehat{b}]) \text{ by definition of $\overline{f}$}\\
&=f(\xi^{\phi_{G_1}(a,b)}) \text{ by definition of  $\phi_{G_1}$}\\
&=\xi^{\phi_{G_1}(a,b)} \text{ since $f$ fixes $C$ by assumption.}
\end{align*}

In particular the module $(A_1,\phi_{G_1})$ is isometric to the module $(A_2,\phi_{G_2})$ since $\overline{f}$ is a group isomorphism. \end{proof}

In the sequel,  given an irreducible subgroup $H$ in $SL(n,\mathbb{C})$, we  apply this proposition to $G:=U_n(H)$, $C:=Z(SL(n,\mathbb{C}))$, $c_0:=\xi I_n$ and $A:=~Z_n(H)$. The construction leads to an alternate module denoted $(Z_n(H),\phi_H)$ and called the \textit{associated alternate module} to $H$. Applying this proposition and corollary \ref{conjugclass}, the isometry class of $(Z_n(H),\phi_H)$ is invariant by conjugation of $Z_n(H)$.

\bigskip
Since we  want to classify conjugacy classes of centralizers of irreducible subgroups in $PSL(n,\mathbb{C})$, we will prove a lemma -a bit more general than we actually need it to be- that will eventually lead to a converse to the second statement in proposition \ref{3.assoc}. Roughly speaking, it states that if we are given two groups with a central cyclic subgroup $C$ such that it leads to isometric alternate modules, then the two groups are isomorphic (provided they verify a condition on the order of the elements).

\begin{lemme}\label{unicitext}

Let $n\geq 1$, $C$ a cyclic group of order $n$ (denoted multiplicatively with a generator $c_0$) and $A$ a finite abelian group. Let $(d_r,\dots,d_1)$ be the type of $A$, then $A$ is isomorphic to $\langle e_r\rangle\times \cdots\times\langle e_1\rangle$ where $e_i$ is an element of $a$ of order $d_i$.  

\bigskip

Let $G$ and $H$ be two central extensions of $A$ by $C$, we denote $\pi_G$ (resp. $\pi_H$) the projection of  $G$ (resp. $H$) on $A$. Assume that $\phi_G=\phi_H$ and for all $1\leq i\leq r$, there exists $g_i\in G$ (resp. $h_i\in H$) such that $\pi_G(g_i)=e_i$ (resp. $\pi_H(h_i)=e_i$) and  the order of $g_i$ is equal to the order of  $h_i$ which is either $d_i$ or $2d_i$. Then, there exists an isomorphism  $f$ between  $G$ and $H$ sending $c_0\in C\leq G$ to $c_0\in C\leq H$.

\end{lemme}

\begin{proof}

We define a set theoretic section $u$ for  $\pi_G$ and $v$ for $\pi_H$. By definition of  $(e_i)$, for any element $a\in A$, there exists a unique  $r$-uple $(\alpha_r,\dots ,\alpha_1)$ of integers such that  :

$$a=\sum_{i=1}^r\alpha_i e_i\text{ and } 0\leq \alpha_i<d_i\text{.}  $$

We define  $u(a):=g_r^{\alpha_r}\dots g_1^{\alpha_1}$ and $v(a):=h_r^{\alpha_r}\dots h_1^{\alpha_1}$. By definition of $g_i$ and $h_i$,  $u$ and $v$ are respectively set-theoretic sections of  $\pi_G$ and $\pi_H$. Any element of  $G$  can be uniquely written as the product of an element $c$ of $C$ and  $u(a)$ where  $a\in A$ and we will define the isomorphism between $G$ and $H$ of the lemma as :

\begin{displaymath}
f:
\left|
  \begin{array}{rcl}
   G& \longrightarrow &H\\
   cu(a)& \longmapsto &cv(a)\\
  \end{array}
\right.\text{.}
\end{displaymath}

The application $f$ is onto, since any element $h$ in the group $H$ can be written as $cv(a)$ and $f(cu(a))=cv(a)=h$, by definition. Furthermore, if $f(c_1u(a_1))=f(c_2u(a_2))$, then $a_1=\pi_H(f(c_1u(a_1)))=\pi_H(f(c_2u(a_2)))=a_2$. It follows that $c_1=c_2$ and hence $c_1u(a_1)=c_2u(a_2)$ so that $f$ is a bijection. Furthermore, $f$ fixes point by point the elements of  $C$ by definition. Remark that those properties of $f$ are easily deduced from the very definition of $f$. We shall need the assumptions of the lemma to show that $f$ is a group morphism.  Let $k_1,k_2\in G$, and write $k_1=c_1u(a_1)$, $k_2=c_2u(a_2)$ where $c_1,c_2\in C$ and $a_1,a_2\in A$. Then  :

\begin{align*}
&f(k_1k_2)= f(k_1)f(k_2)\\
\Leftrightarrow &f(c_1c_2u(a_1)u(a_2))=c_1c_2u(a_1)u(a_2)\\
\Leftrightarrow &f(\underbrace{c_1c_2u(a_1)u(a_2)u(a_1a_2)^{-1}}_{\in C}u(a_1a_2))=c_1c_2v(a_1)v(a_2)\\
\Leftrightarrow &c_1c_2u(a_1)u(a_2)u(a_1a_2)^{-1}v(a_1a_2)=c_1c_2v(a_1)v(a_2)\\
\Leftrightarrow &u(a_1)u(a_2)u(a_1a_2)^{-1}=v(a_1)v(a_2)v(a_1a_2)^{-1}\text{.}
\end{align*}

We need to check that for all  $a_1,a_2\in A$ :

\begin{equation}
u(a_1)u(a_2)u(a_1a_2)^{-1}=v(a_1)v(a_2)v(a_1a_2)^{-1}\text{.}
\label{3.averifier}
\end{equation}

For $i=1,2$, we denote :

$$a_i:=\sum_{j=1}^r\gamma_j^ie_j\text{ where } 0\leq \gamma_j^i<d_j\text{.}  $$

We also write for  $1\leq j\leq r$,  $\delta_j:=\left\lbrace
\begin{array}{ll}
\gamma_j^1+\gamma_j^2&\text{if $\gamma_j^1+\gamma_j^2<d_j$ }\\
\gamma_j^1+\gamma_j^2-d_j&\text{ else\text{.} }
\end{array}\right.$. Then :

$$a_1+a_2=\sum_{j=1}^r\delta_je_j\text{ with } 0\leq \delta_j<d_j\text{.} $$

By definition of $u$ and the expressions of $a_1$ and $a_2$ given above   :

\begin{align*}
u(a_1)u(a_2)&=g_r^{\gamma_r^1}\dots g_1^{\gamma_1^1}g_r^{\gamma_r^2}\dots g_1^{\gamma_1^2}\\
&=g_r^{\gamma_r^1}\dots g_2^{\gamma_2^1}g_r^{\gamma_r^2}g_1^{\gamma_1^1}c_0^{\gamma_1^1\gamma_r^2\phi_G(e_1,e_r)}g_{r-1}^{\gamma_{r-1}^2}\dots g_1^{\gamma_1^2}\text{.}
\end{align*}

We used the equality  $g_rg_1=c_0^{\phi_{G}(e_r,e_1)}g_1g_r$ and the fact $c_0$ is in the center of $G$. By a straightforward induction :

$$u(a_1)u(a_2)=g_r^{\gamma_r^1+\gamma_r^2}g_{r-1}^{\gamma_{r-1}^1}\dots g_1^{\gamma_1^1}g_{r-1}^{\gamma_{r-1}^2}\dots g_1^{\gamma_1^2}\prod_{i=2}^{r-1}c_0^{\phi_{G}(e_i,e_r)\gamma_i^1\gamma_r^2}\text{.}$$

We may do this again for  $g_{r-1}^{\gamma_{r-1}^2}$ up to  $g_{1}^{\gamma_{1}^2}$ and then we have :

$$u(a_1)u(a_2)=g_r^{\gamma_r^1+\gamma_r^2}\dots g_1^{\gamma_1^1+\gamma_1^2}\prod_{j=1}^{r}\prod_{i=1}^{j-1}c_0^{\phi_{G}(e_i,e_j)\gamma_i^1\gamma_j^2} \text{.}$$

Finally remark that for  $1\leq i \leq r$,  $g_i^{\gamma_i^1+\gamma_i^2-\delta_i}$ is in  $C$ whence it is central. So that :

\begin{align*}
g_r^{\gamma_r^1+\gamma_r^2}\dots g_1^{\gamma_1^1+\gamma_1^2}&=g_r^{\gamma_r^1+\gamma_r^2-\delta_r}g_r^{\delta_r}\dots g_1^{\gamma_1^1+\gamma_1^2-\delta_1}g_1^{\delta_1}\\
&=g_r^{\gamma_r^1+\gamma_r^2-\delta_r}\dots g_1^{\gamma_1^1+\gamma_1^2-\delta_1}g_r^{\delta_r}\dots g_1^{\delta_1}\\
&=\prod_{i=1}^rg_i^{\gamma_i^1+\gamma_i^2-\delta_i}u(a_1+a_2)\text{.}
\end{align*}

\begin{equation}
\text{ Finally  : }u(a_1)u(a_2)u(a_1+a_2)^{-1}=\prod_{i=1}^rg_i^{\gamma_i^1+\gamma_i^2-\delta_i}\prod_{j=1}^{r}\prod_{i=1}^{j-1}c_0^{\phi_{G}(e_i,e_j)\gamma_i^1\gamma_j^2} \text{.} 
\label{zu}
\end{equation}

Using the same arguments applied to $H$ :

\begin{equation}
v(a_1)v(a_2)v(a_1+a_2)^{-1}=\prod_{i=1}^rh_i^{\gamma_i^1+\gamma_i^2-\delta_i}\prod_{j=1}^{r}\prod_{i=1}^{j-1}c_0^{\phi_{H}(e_i,e_j)\gamma_i^1\gamma_j^2}\text{.}
\label{zv}
\end{equation}

By assumption, $\phi_G$ and $\phi_H$ are equal, whence :

\begin{equation}
\prod_{i=1}^{j-1}c_0^{\phi_{G}(e_i,e_j)\gamma_i^1\gamma_j^2}=\prod_{i=1}^{j-1}c_0^{\phi_{H}(e_i,e_j)\gamma_i^1\gamma_j^2}\text{.}
\label{psiH}
\end{equation}

Let $1\leq i\leq r$, denote $\lambda_i:=g_i^{\gamma_i^1+\gamma_i^2-\delta_i}$ and $\mu_i:=h_i^{\gamma_i^1+\gamma_i^2-\delta_i}$.  By assumption $g_i$ and $h_i$ have the same order which is either $d_i$ or $2d_i$. If the order of $g_i$ is $d_i$ then, since $\gamma_i^1+\gamma_i^2-\delta_i=0$ or $d_i$, we have that $\lambda_i=1_G=\mu_i$. If $g_i$ and $h_i$ are both of order $2d_i$, then, if $\gamma_i^1+\gamma_i^2-\delta_i=0$, and we have  $\lambda_i=1_G=\mu_i$, otherwise $\gamma_i^1+\gamma_i^2-\delta_i=d_i$, in which case $\lambda_i$ and $\mu_i$ are both elements of order  $2$ in a cyclic group  $C$, but in a cyclic group, there exists at most  one element of order $2$ whence $\lambda_i=\mu_i$.

\bigskip

In any case $h_i^{\gamma_i^1+\gamma_i^2-\delta_i}=g_i^{\gamma_i^1+\gamma_i^2-\delta_i}$ for  $1\leq i \leq r$. It follows that :

$$\prod_{i=1}^rg_i^{\gamma_i^1+\gamma_i^2-\delta_i}=\prod_{i=1}^rh_i^{\gamma_i^1+\gamma_i^2-\delta_i}\text{.}$$

Combining this equality to the equations \ref{zu},\ref{zv}  and \ref{psiH}, we get  for $a_1,a_2\in A$ :

$$u(a_1)u(a_2)u(a_1+a_2)^{-1}=v(a_1)v(a_2)v(a_1+a_2)^{-1}\text{.}$$

Therefore, $f$ is a group morphism and the lemma is proven.\end{proof}

We are now ready to state our first theorem.

\begin{thm}\label{classif1}

Let $n\geq 1$, $\overline{H_1}$ and $\overline{H_2}$  be two irreducible subgroups of $PSL(n,\mathbb{C})$. Denote $H_i:=\pi_n^{-1}(\overline{H_i})$ the pull-back in $SL(n,\mathbb{C})$. Then $Z_{PSL(n,\mathbb{C})}(\overline{H_1})$ is conjugate to $Z_{PSL(n,\mathbb{C})}(\overline{H_2})$ if and only if $(Z_{PSL(n,\mathbb{C})}(\overline{H_1}),\phi_{H_1})$ is isometric to $(Z_{PSL(n,\mathbb{C})}(\overline{H_2}),\phi_{H_2})$.

\end{thm}

\begin{proof}
First, remark that for $i=1,2$, $Z_{PSL(n,\mathbb{C})}(\overline{H_i})=Z_n(H_i)$. If $Z_n(H_1)$ and $Z_n(H_2)$, proposition \ref{3.assoc} and corollary \ref{conjugclass} imply that $(Z_n(H_1),\phi_{H_1})$ and  $(Z_n(H_2),\phi_{H_2})$ are isometric.

\bigskip
Conversly, say   $(Z_n(H_1),\phi_{H_1})$ and $(Z_n(H_2),\phi_{H_2})$ are isometric. Then, up to composing the projection $\pi_1:U_n(H_1)\rightarrow Z_n(H_1)$ by the isometry between   $(Z_n(H_1),\phi_{H_1})$ and $(Z_n(H_2),\phi_{H_2})$, we may assume that the alternate modules constructed from $U_n(H_i)$ are both equal to $(A,\phi)$. Let $a$ be in $A$ and $d$ be its order, then, applying proposition \ref{3.diagonal}, there exists $u_i\in U_n(H_i)$ such that :

$$u_i\text{ is conjugate to }\lambda \begin{pmatrix}I_{\frac{n}{d}}&&&\\&\xi^{\frac{n}{d}}I_{\frac{n}{d}}&&\\&&\ddots&\\&&&\xi^{\frac{n}{d}(d-1)}I_{\frac{n}{d}}\end{pmatrix} $$

$$\text{ with }\lambda\in \left\lbrace
\begin{array}{ll}
\langle \xi I_n\rangle  &   \text{if $d$ is odd or $d$   even and $n/d$ even} \\
\sqrt{\xi^{-\frac{n(d-1)}{d}}}\langle \xi I_n\rangle& \mbox{if $d$ is even and $n/d$ odd}\text{.} 
\end{array}
\right.$$

Up to multiplying $u_i$ by $\xi^k$, we may take $\lambda=1$ if $d$ is odd or $n/d$ is even (in which case the lift $u_i$ of $a$ is of order $d$) and $\lambda=\sqrt{\xi^{-\frac{n(d-1)}{d}}}$ else (in which case the lift $u_i$ of $a$ is of order $2d$). We are, now, in the conditions of application of lemma~\ref{unicitext}  and there exists a group isomorphism between $U_n(H_1)$ and $U_n(H_2)$ fixing, point by point $\langle \xi I_n\rangle$. Hence $Z_n(H_1)$ and $Z_n(H_2)$ are conjugate by corollary \ref{conjugclass}. \end{proof}

 In the next section, we shall see under which necessary and sufficient condition  an alternate module $(A,\phi)$ is associated to a centralizer of an irreducible subgroup in  $PSL(n,\mathbb{C})$.

\section{Classification of   centralizers in $PSL(n,\mathbb{C})$}\label{sec3}

In this section, we characterize the alternate modules associated to centralizers of irreducible subgroups in $PSL(n,\mathbb{C})$. Recall first that applying corollary \ref{3.ncardinal}, if an abelian group is isomorphic to such centralizer then its order divides $n^2$.  A centralizer of an irreducible subgroup in $PSL(n,\mathbb{C})$ is   \textit{full} if its order is $n^2$.

In the first subsection, we give a necessary and sufficient condition for an alternate module to be associated to such centralizer. Full centralizers play a key role in the study. In the second section, we focus on the consequences of this result.

\subsection{$n$-subsymplectic modules and associated centralizers}\label{sec3sub1}

Once again, we begin with a result from representation theory. Let $G$ be a finite group and $\rho:G\rightarrow GL(n,\mathbb{C})$ be a representation of $G$ acting linearly on $V^{\mathbb{Z}/n}$ through $\rho$ then $V$ may be decomposed as a sum of maximal isotypic sub-modules $V_1$, \dots, $V_k$ (i.e. which are linearly equivalent to $\lambda_i\cdot(W_i,\rho_i)$ where $(W_i,\rho_i)$ is an irreducible representation of $G$ and $\lambda_i>0$). Furthermore, up to the order of $(V_i)$, the decomposition happens to be unique. This leads to a technical lemma :
\begin{lemme}\label{isotypic}

Let $n\geq 1$, $G$ be a finite subgroup of $GL(n,\mathbb{C})$. Define $\iota_G$ to be the inclusion of $G$ into $GL(n,\mathbb{C})$. Let  $V_1\oplus\cdots\oplus V_k$ be the maximal isotypic decomposition of $(V,\iota_G)$.  Then $N_{GL(n,\mathbb{C})}(G)$ acts on the set $\{V_1,\dots,V_k\}$ of subspaces occuring in the maximal isotypic decomposition of $(V,\iota_G)$.

\end{lemme}

\begin{proof}

Let $n$ be in $N_{GL(n,\mathbb{C})}(G)$ and $1\leq i\leq k$. For all $g\in G$ and $v\in V_i$ :

$$g\cdot (n(v_i))=(gn)\cdot v_i=n(n^{-1}gn)\cdot(v_i)=n\overbrace{\underbrace{n^{-1}gn}_{\in G}\cdot(v_i)}^{\in V_i}\text{.}$$

As a result $nV_i$ is stable by $G$ and $nV_i$ is a submodule of $V$. Furthermore, if $(V_i,\rho_i')$ is the induced representation  on $V_i$ by $\iota_G$, then the representation on $nV_i$ induced by $\iota_G$ is $(nV_i,\rho_i'(n^{-1} \cdot n))$. In particular, $V_i$ being isotypic,   $nV_i$ is also isotypic and is therefore included in a maximal isotypic submodule.  This implies that there exists $1\leq j\leq r$ such that  $nV_i\leq V_j\Leftrightarrow V_i\leq n^{-1}V_j $.

Using the same argument, $ n^{-1}V_j $ is also isotypic, but it contains $V_i$ which is maximal isotypic, it follows that $V_i=n^{-1}V_j$ whence $nV_i=V_j$. \end{proof}

We recall   notations and  results from \cite{Gue2}. Let $(A,\phi)$ be an alternate module, we say that $K$ included in $A$ is \textit{isotropic} if $K$ is orthogonal to itself. We say that $L$ included in  $A$ is \textit{Lagrangian} if $L^{\perp}=L$. Remark that (proposition 2 in loc. cit.) Lagrangians are exactly the maximal isotropic subsets in $A$. Likewise, we have proven that any Lagrangian of the alternate module $(A,\phi)$ is of order $n_{A,\phi}:=\sqrt{|A||K_{\phi}|}$ where $K_{\phi}$ denotes the radical (or kernel) of $(A,\phi)$.

\begin{prop}\label{cardlag}

Let $n\geq 1$ and  $H$ be an irreducible subgroup of $SL(n,\mathbb{C})$, then  the order $n_{Z_n(H),\phi_H}$ of Lagrangians in $(Z_n(H),\phi_H)$ divides $n$.

\end{prop}
 
\begin{proof}

Let $K$ be an isotropic subgroup in $(Z_n(H),\phi_H)$ and $K_0$  be its pull-back in $U_n(H)$. If $k_1,k_2\in K_0$ then :

\begin{align*}
\text{}[k_1,k_2]&=\xi^{\phi_H(\pi_n(k_1),\pi_n(k_2))}\text{ by definition of $\phi_H$}\\
&=I_n\text{ since $(K,\phi_{H|K\times K})$ is trivial}\text{.} 
\end{align*}

Therefore, if $K$ is isotropic, then $\pi_n^{-1}(K)$ is an abelian subgroup in $U_n(H)$. In particular, if $L$ is a Lagrangian in $(Z_n(H),\phi_H)$, then $L_0:=\pi_n^{-1}(L)$ is an abelian subgroup of $U_n(H)$.  Denote $\iota_0: L_0\rightarrow SL(n,\mathbb{C})$   the inclusion of $L_0$ in $SL(n,\mathbb{C})$. Let $h\in H$ and $x\in L_0$ then $hxh^{-1}\in \langle \xi I_n\rangle x$ since $x\in U_n(H)$. Because $\langle \xi I_n\rangle$ is included in $L_0$, it follows that $hxh^{-1}$ belongs to  $L_0$ as well. Whence,  $H$ normalizes $L_0$.

\bigskip

We denote $V_1$, \dots, $V_k$ to be the maximal isotypic subspaces occuring in $(V,\iota_0)$. Using lemma \ref{isotypic}, $H$ acts on $\{V_1,\dots,V_k\}$. Furthermore, if the action were not transitive then $H$ would stabilize a non-trivial subspace of $V$  contradicting its irreducibility. Whence, the action of $H$ on $\{V_1,\dots,V_k\}$ is transitive and, in particular, $V_1$,\dots, $V_k$ share the same dimension $\lambda>0$. 

\bigskip

Since $L_0$ is an abelian group, each  $V_i$ can be decomposed as $\lambda$ copies of a $1$-dimensional representation $(W_i,\rho_i)$ of $L_0$. We denote $\chi_0$ to be the character of $\iota_0$ and $\chi_i $ the character of $(V_i,\iota_{0,|V_i})$. By considerations of finite groups representation theory (see \cite{Ser}), $\|\chi_i\|^2=\lambda^2$ and by orthogonality of the characters  $\chi_i$ and $\chi_j$ for $i\neq j$ :

\begin{align*}
\|\chi_0\|^2=k\lambda^2=\lambda \underbrace{\dim(V)}_{=n}\text{.} 
\end{align*}

On the other hand, $\chi_0=\chi_{H| L_0}$ where $\chi_H$ is the character of the standard representation $\iota_H$ of $U_n(H)$ whose character has been computed in proposition \ref{3.car}. Hence~:

$$\|\chi_0\|^2=\frac{n^3}{|L_0|}\text{.} $$

As a result, we have that $|L_0|=n\frac{n}{\lambda}$. Since $L_0=\pi_n^{-1}(L)$ and the kernel of $\pi_n$ is of order $n$, we end up with $|L|=\frac{n}{\lambda}$. In particular the order of $L$, which is $n_{Z_n(H),\phi_H}$, divides $n$.  \end{proof}

We just found a necessary condition (on the cardinality of Lagrangians) for an alternate module to be associated to the centralizer of an irreducible subgroup of  $PSL(n,\mathbb{C})$. It will be proven to be a sufficient condition as  well in theorem \ref{classif2}. In order to prove this, we need to construct some particular examples of irreducible subgroups in $PSL(n,\mathbb{C})$.

\begin{prop}\label{examp}
For any $n\geq 1$ and $B$  abelian group of order $n$, there exists a finite irreducible subgroup $H$ of $SL(n,\mathbb{C})$ of order $n^2$ such that $Z_n(H)$ is isomorphic  to $B\times B$.
\end{prop}

\begin{proof}
We prove it by induction on $n\geq 1$. If $n=1$, then the property is obviously true. Assume $n>1$. Let $B$ be an abelian group of order $n$ and $d$ be the exponent of $B$. Denote $B:=B_1\times \langle e\rangle$ where $e$ is of order $d$.  Let $K$ be an irreducible subgroup of $SL(n/d,\mathbb{C})$ of order $(n/d)^2$  such that $Z_{n/d}(K)$ is isomorphic to $B_1$ (by induction hypothesis). We define a subgroup $K_0$ of $SL(n,\mathbb{C})$ by blocks of size $n/d$ :

$$K_0:=\left\lbrace\begin{pmatrix} k&&\\&\ddots&\\&&k\end{pmatrix}\mid  k\in K\right\rbrace\text{.} $$

We also denote :

$$M:=\lambda \begin{pmatrix}0&&&I_{n/d}\\I_{n/d}&\ddots&\\&\ddots&\ddots&\\&&I_{n/d}&0\end{pmatrix}\text{ and } u:=\lambda \begin{pmatrix}I_{\frac{n}{d}}&&&\\&\xi^{\frac{n}{d}}I_{\frac{n}{d}}&&\\&&\ddots&\\&&&\xi^{\frac{n}{d}(d-1)}I_{\frac{n}{d}}\end{pmatrix}  $$

where $\lambda$ is defined as in proposition \ref{3.diagonal} (this implies that $\det(M)=\det(u)=1$). Finally, let $H$ be the subgroup of $SL(n,\mathbb{C})$ generated by $K_0$, $M$ , $u$ and $\xi I_n$.  

Since $u$ is scalar by blocks of size $n/d$ with distinct eigenvalues, it follows that :

$$Z_{SL(n,\mathbb{C})}(u)=\left\lbrace\begin{pmatrix}g_0&&\\&\ddots&\\&&g_{d-1}\end{pmatrix}  \left| \begin{array}{l}g_0,\dots,g_{d-1}\in GL(n/d,\mathbb{C}) \\ \det(g_0)\dots \det(g_{d-1})=1\end{array}\right.\right\rbrace \text{.} $$

The commutator $[M,u]$ is of order $d$ in $Z(SL(n,\mathbb{C}))$ which is the order of $\pi_n(u)$. As a result,  $U_n(H)\leq U_n(\langle u\rangle)=\langle Z_{SL(n,\mathbb{C})}(u), M\rangle $. Let $g_0,\dots, g_{d-1}$ be in $GL(n/d,\mathbb{C})$ and $b\geq 0$ such that :

$$x:=\begin{pmatrix}g_0&&\\&\ddots&\\&&g_{d-1}\end{pmatrix}M^b\in U_n(H) \text{.}  $$

Then $[M,x]$ is in $Z(SL(n,\mathbb{C}))$ and since $\pi_n(M)$ is of order $d$, it follows that $[M,x]$ is of order dividing $d$. Since $[M,u]$ is of order $d$, up to multiplying $x$ by  some power of $u$, we may assume that $[M,x]=I_n$. But this  implies (by a straightforward calculation) that $g_0=\cdots=g_{d-1}$. Hence we just showed that :

$$U_n(H)\leq \left\lbrace\begin{pmatrix}g&&\\&\ddots&\\&&g\end{pmatrix}u^lM^b \left|\begin{array}{l}g\in GL(n/d,\mathbb{C})\text{, } \det(g)^d=1\\b,l\geq 0\end{array}\right.\right\rbrace  \text{.}$$

$$\text{In the sequel, $D(g)$ will be the matrix } \begin{pmatrix}g&&\\&\ddots&\\&&g\end{pmatrix}\text{ where } g\in  GL(n/d,\mathbb{C})\text{.}$$

Let $g\in GL(n/d,\mathbb{C})$ with $\det(g)^d=1$ then $\det(g)$ is a $d$-th root of the unity. Up to multiplying $D(g)$ by a $n/d$-th root of $\det(g)$ (which is a $n$-th root of the unity, i.e. $\xi^s$ for some integer $s$), we may assume that $\det(g)=1$. Whence :

\begin{equation}
U_n(H)\leq \left\lbrace\xi^sD(g)u^lM^b \left|\begin{array}{l}g\in SL(n/d,\mathbb{C}) \\ s,b,l\geq 0\end{array}\right.\right\rbrace \text{.}
\label{decompH}
\end{equation}

Let us define an application  $\psi_0: \left|\begin{array}{rcl}
U_n(H)&\longrightarrow&\langle \xi^{n/d}\rangle\times \langle \xi^{n/d}\rangle\\
v&\longmapsto& ([v,M],[v,u])\end{array}\right.
$.

\bigskip

 The application $\psi_0$ is a group morphism (consequence of lemma \ref{2l1}), it is surjective since $\psi_0(u)$ and $\psi_0(M)$ generate $\langle \xi^{n/d}\rangle\times \langle \xi^{n/d}\rangle$ and it  factors through $Z_n(H)$. It follows that we have a surjective isomorphism $\psi$ from $Z_n(H)$ to $\mathbb{Z}/d\times \mathbb{Z}/d$.

\bigskip

From the inclusion \ref{decompH}, it follows that $Ker(\psi)\leq \left\lbrace\pi_n(D(g))\mid g\in SL(n,\mathbb{C})\right\rbrace $. In particular $\langle \pi_n(M),\pi_n(u)\rangle$ is of trivial intersection with $Ker(\psi)$. Since it surjects onto $\mathbb{Z}/d\times \mathbb{Z}/d$ through $\psi$, it follows that :

\begin{equation}
Z_n(H)\text{ is isomorphic to }Ker(\psi)\times \mathbb{Z}/d\times \mathbb{Z}/d\text{.} 
\label{decompZn}
\end{equation}

Let   $k\in K$ and  $k_0:=D(k)\in K_0\leq H$. For  $g\in SL(n/d,\mathbb{C})$ : $\left[k_0, D(g) \right]=D([k,g])$. As a result, $D(g)\in U_n(H)$ if and only if $g\in U_{n/d}(K)$. In particular, this leads to a well defined group morphism modulo $\langle \xi I_n\rangle$ :

$$\varphi : \left|\begin{array}{rcl}
Ker(\psi)&\longrightarrow&  Z_{n/d}(K)\\
\pi_n(D(g))&\longmapsto& \pi_{n/d}(g)\end{array}\right.\text{.} 
$$

Because $\varphi$ is  injective, $Ker(\psi)$ is isomorphic to $Z_{n/d}(K)$ which has been chosen to be isomorphic to $B_1\times B_1$. As a result, with the decomposition \ref{decompZn}, $Z_n(H)$ is isomorphic to $B_1\times B_1\times \mathbb{Z}/d\times \mathbb{Z}/d$ which is isomorphic to $B\times B$. Remark that $|H|=|K|d^2=n^2$. In particular, $H$ is finite so it is completely reducible and since its centralizer is of finite order, it is irreducible. So we can conclude by induction.\end{proof}

In the proof, we did not only find an irreducible subgroup $\overline{H}$ of $PSL(n,\mathbb{C})$ whose centralizer is isomorphic to $B\times B$ but, as is, the induction also shows that $\overline{H}$ is its own centralizer. We did not emphasize it in the proposition because, in the next proposition, we shall see that full centralizers (i.e. $Z_n(H)$ of order $n^2$ where $H$ is an irreducible subgroup of $SL(n,\mathbb{C})$)   have many properties including this one. 

\bigskip

We recall (see \cite{T-W}) that an alternate module $(A,\phi)$ is said to be \textit{symplectic} if its radical is trivial. Furthermore if $(A,\phi)$ is symplectic then $(A,\phi)$ is isometric to $B\times B^*$ (corollary 7.4 in loc. cit.) where $B$ is abelian of order $n$ and $B\times B^*$ is endowed with its canonical symplectic form.

\begin{prop}\label{full}

Let $n\geq 1$ and $H$ be an irreducible subgroup of $SL(n,\mathbb{C})$ containing the center of $SL(n,\mathbb{C})$. Then  the following assertions are equivalent :

\begin{displaymath}
\begin{array}{ll}
1&|Z_n(H)|=n^2\\
2&Z_n(H)\text{ is irreducible.}\\
3& \pi_n(H)\text{ is abelian.}\\
4& \pi_n(H)\text{ is its own centralizer.}\\
5& (Z_n(H),\phi_H)\text{ is isometric to $B\times B^*$ where $B$ is abelian of order $n$}
 \end{array}
\end{displaymath}

Furthermore, if $H_1$ and $H_2$ are two irreducible subgroups verifying one of these assertions and  $Z_n(H_1)$ is abstractly isomorphic to $Z_n(H_2)$ then $H_1$ and $H_2$ are conjugate.

\end{prop}

\begin{proof}
The equivalence $1\Leftrightarrow 2$ is exactly the second assertion in corollary \ref{3.ncardinal}.

\bigskip

($2\Rightarrow 3$). Say $Z_n(H)$ is irreducible, then its centralizer is abelian by proposition \ref{2p1}, therefore, $\pi_n(H)$ which commutes with $Z_n(H)$ is abelian as well.

($3\Rightarrow 4$). Let $\rho$ be the inclusion of $H$  into $SL(n,\mathbb{C})$ and $\chi$ its character. Since $H$ is included in  $U_n(H)$, $\rho$ is the standard representation $\iota_H$ of $U_n(H)$ restricted to $H$ whose character has been computed in proposition \ref{3.car}. It follows that $\|\chi\|^2=\frac{n^3}{|H|} $. On the other hand, since $H$ is irreducible, $\rho$ is irreducible and then $\|\chi^2\|=1$. Whence $|H|=n^3$ and $|\pi_n(H)|=n^2$. Since $|Z_n(H)|\leq n^2$ and $\pi_n(H)$ is included in $Z_n(H)$, it follows that $\pi_n(H)=Z_n(H)$ and $\pi_n(H)$ is its own centralizer.

($4\Rightarrow 2$). If $\pi_n(H)$ is its own centralizer then $Z_n(H)=\pi_n(H)$ is irreducible.

\bigskip

($5\Rightarrow 1$) is obvious. ($1\Rightarrow 5$). The alternate module $(Z_n(H),\phi_H)$ is of order $n^2$. By proposition \ref{cardlag}, we  have that $n_{Z_n(H),\phi_H}:=\sqrt{|Z_n(H)||K_{\phi_H}|}$ divides $n$. It follows that $\sqrt{|K_{\phi_H}|}$ divides $1$. Whence $K_{\phi_H}$ is trivial and $(Z_n(H),\phi_H)$ is a symplectic module of order $n^2$. By corollary 7.4  in \cite{T-W}, there exists an abelian group $B$ of order $n$ such that $(A,\phi)$ is isometric to $B\times B^*$ with its canonical symplectic form.

\bigskip

Let $H_1$ and $H_2$ be two irreducible subgroups in $SL(n,\mathbb{C})$ containing $Z(SL(n,\mathbb{C}))$ such that $Z_n(H_1)$ and $Z_n(H_2)$ are isomorphic. Since they verify the assertion $5$, $(Z_n(H_1),\phi_{H_1})$ and $(Z_n(H_2),\phi_{H_2})$ are both symplectic modules with isomorphic underlying groups. Since the isometry class of a symplectic module is simply given by the isomorphism class of its underlying group, it follows that $(Z_n(H_1),\phi_{H_1})$ and $(Z_n(H_2),\phi_{H_2})$ are isometric. As a result, applying theorem \ref{classif1}, the subgroups  $Z_n(H_1)$ and $Z_n(H_2)$ are conjugate in $PSL(n,\mathbb{C})$. Since for $i=1,2$, $\pi_n(H_i)=Z_n(H_i)$ and $H_i=\pi_n^{-1}(\pi_n(H_i))$ it follows that $H_1$ is conjugate to $H_2$. \end{proof}

We end this subsection with a characterization theorem :
 
\begin{thm}\label{classif2}

Let $n\geq 1$ and $(A,\phi)$ be an alternate module, then the following assertions are equivalent :

1.There exists an irreducible subgroup $H$ in $SL(n,\mathbb{C})$ such that $(Z_n(H),\phi_H)$ is isometric to $(A,\phi)$.

2. The order of Lagrangians in $(A,\phi)$ divides $n$.

3. There exists an abelian group $B$ of order $n$ such that $(A,\phi)$ is isometrically embedded in $B\times B^*$.

In particular,  full centralizers of irreducible subgroups in $SL(n,\mathbb{C})$ are the maximal elements among the centralizers of irreducible subgroups in $SL(n,\mathbb{C})$.

\end{thm}
\begin{proof}
The implication $1\Rightarrow 2$ is exactly the statement of proposition \ref{cardlag}.

\bigskip

Let us show that $2\Rightarrow 3$. Let $(A,\phi)$ be an alternate module and $d=n_{A,\phi}$ be the order of one of its Lagrangians which divides $n$ by assumption. Using the theorem proven in \cite{Gue2} (alternate modules are subsymplectic), there exists an abelian group $B_0$ of order $d$ such that $(A,\phi)$ is isometrically embedded in $B_0\times B_0^*$. Defining the abelian group $B:=B_0\times \mathbb{Z}/(n/d)$,  $B_0\times B_0^*$ is isometrically embedded in $B\times B^*$, so that  $(A,\phi)$ is also isometrically embedded in $B\times B^*$. Since $|B|=|B_0|n/d=n$, we have assertion $3$. 

\bigskip

Let us show that $3\Rightarrow 1$. Assume that $(A,\phi)$ is isometrically embedded in $B\times B^*$.  Let $H$ be an irreducible subgroup such that $Z_n(H)$ is isomorphic (as a group) to $B\times B$ (by proposition \ref{examp}). Since $|Z_n(H)|=|B|^2=n^2$, it follows, by proposition \ref{full}, that $(Z_n(H),\phi_H)$ is isometric to $B\times B^*$. We denote $f$ a group isomorphism from $Z_n(H)$ to $B\times B^*$ which is an isometry. 

\bigskip

Denote $(A_0,\phi_0)$ the submodule of $(Z_n(H),\phi_H)$ which is sent to $(A,\phi)$ by $f$. By corollary \ref{3.stabparsgrp}, $A_0$ is the $n$-centralizer of an irreducible subgroup $K$ of $SL(n,\mathbb{C})$, that is to say $Z_n(K)=A_0$. Furthermore, from the definition of $(A_0,\phi_K)$ and $(Z_n(H),\phi_H)$,  $\phi_{H|A_0\times A_0}=\phi_K$. Since $\phi_0:=\phi_{H|A_0\times A_0}$, the alternate module $(A_0,\phi_0)$ is  of the form $(Z_n(K),\phi_K)$. Since $(A,\phi)$ is  isometric to $(A_0,\phi_0)$ we have the assertion $1$. 

\bigskip

By corollary \ref{3.ncardinal}, full centralizers reach the maximal order of centralizers of irreducible subgroups in $SL(n,\mathbb{C})$. Whence, full centralizers are maximal among the centralizers of irreducible subgroups in $SL(n,\mathbb{C})$. 

\bigskip

Conversly, if $H$ is an irreducible subgroup in $SL(n,\mathbb{C})$, then $(Z_n(H),\phi_H)$ is isometrically embedded in $B\times B^*$ (since $1\Leftrightarrow 3$) with $B$ abelian group of order $n$. Using propositions \ref{examp} and \ref{full}, there exists an irreducible subgroup $S$  of $PSL(n,\mathbb{C})$ such that $S$ is isomorphic to $B\times B^*$. If $A$ is the image of $Z_n(H)$ into $B\times B^*$ then $A=Z_n(K)$ by corollary  \ref{3.stabparsgrp}. It follows that $(Z_n(H),\phi_H)$ is isometric to $(Z_n(K),\phi_K)$. Hence $Z_n(H)$ and $Z_n(K)$ are conjugate (by theorem \ref{classif1}), since $Z_n(K)$ is included in $S=Z_{PSL(n,\mathbb{C})}(S)$ which is a full centralizer, we have that any centralizer of irreducible subgroup in $PSL(n,\mathbb{C})$ is contained in a full centralizer. This implies that all maximal elements among  the centralizers of irreducible subgroups in $SL(n,\mathbb{C})$ must be full centralizers. \end{proof}

This theorem will  have some consequences that we are going to sum up in the next subsection. In the sequel, if $(A,\phi)$ is an alternate module, we will say that $(A,\phi)$ is $n$-\textit{subsymplectic} if there exists an abelian group $B$ of order $n$ such that $(A,\phi)$ is isometrically embedded in $B\times B^*$.

\subsection{Consequences of the classification}\label{sec3sub2}

Something that may be the most obvious consequence is the following :

\begin{cor}\label{peudemod}

Let $n\geq 1$, then the number of conjugacy classes of centralizers of irreducible subgroups in $PSL(n,\mathbb{C})$ is finite.

\end{cor}

\begin{proof}
Using theorem \ref{classif1}, we can equivalently bound the number of isometry classes of  associated alternate modules. Using theorem \ref{classif2}, it suffices to show that  there are only a finite number of isometry classes of $n$-subsymplectic modules.

\bigskip

When we fix $B$, the module $B\times B^*$ has only a finite number of submodules, since there exist only a finite number of abelian groups $B$ of order $n$, there are only a finite number of symplectic modules of order $n^2$, up to isometry, and hence, a finite number of $n$-subsymplectic modules up to isometry. \end{proof}

In the next corollary, we characterize the isomorphism classes of centralizers :

\begin{cor}\label{ncompatible}
Let $n\geq 1$ and $A$ be an abelian group. Then there exists an irreducible subgroup $H$ of $SL(n,\mathbb{C})$ such that $Z_n(H)$ is isomorphic to $A$ if and only if there exists an abelian group $B$ of order $n$ such that $A$ is isomorphic to a subgroup of $B\times B$. 

In particular, for any abelian group $A$, there exists an integer $n\geq 1$ and an irreducible subgroup $H$ of $SL(n,\mathbb{C})$ such that $Z_n(H)$ is isomorphic to $A$.

\end{cor}

\begin{proof}
If $A$ is isomorphic to $Z_n(H)$ where $H$ is an irreducible subgroup of $SL(n,\mathbb{C})$ then, by theorem \ref{classif2}, there exists an abelian group $B$ of order $n$ such that $(Z_n(H),\phi_H)$ is isometrically embedded in $B\times B^*$. In particular $Z_n(H)$ is isomorphic to a subgroup of $B\times B^*$ which is (non-canonically) isomorphic to $B\times B$. Hence $A$ is isomorphic to a subgroup of $B\times B$. 

\bigskip
Let $B$ be  an abelian group $B$ of order $n$ such that $A$ is isomorphic to a subgroup of $B\times B$. By proposition \ref{examp}, there exists  an irreducible subgroup $H$ of $SL(n,\mathbb{C})$ such that $Z_n(H)=B\times B$. Let $A_0$ be the image of $A\leq B\times B$ in $Z_n(H)$, by corollary \ref{3.stabparsgrp}, $A$ is the $n$-centralizer of an irreducible subgroup of $SL(n,\mathbb{C})$.

\bigskip

Remark that $A$ is always included in $A\times A$. Hence, if $n:=|A|$, applying what we have already done in this corollary, there exists an irreducible subgroup $H$ of $SL(n,\mathbb{C})$ such that $A$ is isomorphic to $Z_n(H)$.  \end{proof}

If $p$ is a prime number, we have recalled in the introduction  that conjugacy classes of centralizers of irreducible subgroups in $PSL(p,\mathbb{C})$  are classified by their isomorphism classes. We would like to know exactly when this   convenient property holds. We recall that an integer $n\geq 2$ is \textit{squarefree} if it cannot be divided by a non-trivial square.

\begin{lemme}\label{modsqf}

Let $n$ be a squarefree integer and $(A,\phi)$  be   an alternate module  of rank $2$, of exponent dividing $n$ such that $n_{A,\phi}$ divides $n$. Let $A$ be isomorphic to $\langle b\rangle\times \langle a\rangle$ with $e$ (the order of $b$) dividing $d$  (the order of $a$) dividing $n$. Then :

$$(A,\phi)\text{ is isometric to } \left(\mathbb{Z}/e\times\mathbb{Z}/d,\begin{pmatrix}0&-1/e\\1/e&0\end{pmatrix}\right)\text{.} $$

\end{lemme}

\begin{proof}

Let $e'$ be the order of $\phi(b,a)$.  Then $K_{\phi}$ contains $\langle b^{e'}\rangle$ of order $e/e'$ and $\langle a^{e'}\rangle$ of order $d/e'$. Since those two groups are in trivial intersection, $K_{\phi}$ contains their direct product and $|K_{\phi}|$ is divided by $\frac{de}{(e')^2}$.  Remark that $|A|=de$. It follows that

$$n_{A,\phi}:=\sqrt{|A||K_{\phi}|} =de/e'\text{.} $$

Since $n_{A,\phi}$ divides $n$ by assumption, $d\frac{e}{e'}$ divides $n$. Assume that $e'$ is a strict divisor of $e$ and take $p$ a prime number dividing $\frac{e}{e'}$ then $p$ divides $d$ and, hence, $p^2$ divides $n$ which contradicts the fact that $n$ is squarefree. Hence $\phi(b,a)$ is of order $e$. 

\bigskip

Denote $\phi(b,a)=\frac{\lambda}{e}\in\mathbb{Q}/\mathbb{Z}$. Since $\phi(b,a)$ is of order $e$, $\lambda$ is prime to $e$. Let $\mu$ be an integer such that $\lambda\mu=1$ mod $e$. Then the automorphism $f:A\rightarrow A$ sending $b$ to $b^{\mu}$ and $a$ to $a$ is an isometry between $(A,\phi)$ and $\left(\mathbb{Z}/e\times\mathbb{Z}/d,\begin{pmatrix}0&-1/e\\1/e&0\end{pmatrix}\right)$. \end{proof}

This lemma leads to the following theorem :

\begin{thm}\label{sqf}

Let $n\geq 2$. Then  conjugacy classes of centralizers of irreducible subgroups of $PSL(n,\mathbb{C})$ are classified by their isomorphism classes if and only if $n$ is squarefree.

\end{thm}

\begin{proof}
Assume that $n$ is squarefree. Let $H_1$ and $H_2$ be two irreducible subgroups in $SL(n,\mathbb{C})$ such that $Z_n(H_1)$ and $Z_n(H_2)$ are isomorphic. First remark that, by corollary~\ref{ncompatible}, $Z_n(H_i)$ is necessarily isomorphic to a subgroup of $B\times B$ where $B$ is abelian of order $n$. Since $n$ is squarefree, $B$ is actually cyclic. It follows that the common rank of $Z_n(H_1)$ and $Z_n(H_2)$ is lesser or equal to $2$.

If the common rank of $Z_n(H_1)$ and $Z_n(H_2)$ is $0$ or $1$, then for $i=1,2$,  $(Z_n(H_i),\phi_{H_i})$ is necessarily the trivial module (i.e. $\phi_{H_i}=0$). Since $Z_n(H_1)$ and $Z_n(H_2)$ are isomorphic, it is clear that $(Z_n(H_1),\phi_{H_1})$ and $(Z_n(H_2),\phi_{H_2})$ are isometric. Whence, using theorem~\ref{classif1}, $Z_n(H_1)$ and $Z_n(H_2)$ are conjugate.

 Else, the common rank of $Z_n(H_1)$ and $Z_n(H_2)$ is $2$. Let $\mathbb{Z}/e\times\mathbb{Z}/d$ be the isomorphism class of both $Z_n(H_1)$ and $Z_n(H_2)$. Since $n_{Z_n(H_1),\phi_{H_1}}$ and $n_{Z_n(H_2),\phi_{H_2}}$ both divide $n$, we are in the condition of application of lemma \ref{modsqf}. In particular $(Z_n(H_1),\phi_{H_1})$ and $(Z_n(H_2),\phi_{H_2})$ are isometric and, by theorem \ref{classif1}, $Z_n(H_1)$ and $Z_n(H_2)$ are conjugate. 

\bigskip

Assume that $n$ is not squarefree. Let $p$ be a prime number such that $p^2$ divides $n$. We define two structures of alternate module on $\mathbb{Z}/p\times \mathbb{Z}/p$ :

\begin{align*}
M_1:=(\mathbb{Z}/p\times\mathbb{Z}/p,0)& \text{ where the form is trivial}\text{.} \\
M_2:=\mathbb{Z}/p\times(\mathbb{Z}/p)^*&\text{ where the form is symplectic}\text{.} 
\end{align*}

Remark that $M_1$ is Lagrangian in $M_1$ so the order of its Lagrangian is $p^2$ which divides $n$. By theorem \ref{classif2}, there exists an irreducible subgroup $H_1$ of $SL(n,\mathbb{C})$ such that $(Z_n(H_1),\phi_{H_1})$ is isometric to $M_1$.  On the other hand the radical of $M_2$ is trivial, so the order of its Lagrangian is $p$ which divides $n$. By theorem \ref{classif2}, there exists an irreducible subgroup $H_2$ of $SL(n,\mathbb{C})$ such that $(Z_n(H_2),\phi_{H_2})$ is isometric to $M_2$.

Even if $Z_n(H_1)$ and $Z_n(H_2)$ are both isomorphic to $\mathbb{Z}/p\times\mathbb{Z}/p$, $Z_n(H_1)$ and $Z_n(H_2)$ are not conjugate because $M_1$ and $M_2$ are  not isometric  (see the last assertion of proposition~\ref{3.assoc}). \end{proof}

As a result, when $n$ is squarefree, we can compute the number of conjugacy classes of centralizers of irreducible subgroups in $PSL(n,\mathbb{C})$. 

\begin{cor}\label{sqfnumber}

Let $n=p_1\dots p_r$ be a squarefree integer (i.e. $p_1$,\dots, $p_r$ are two-by-two different prime numbers). There are exactly $3^r$ conjugacy classes of centralizers of irreducible subgroups in $PSL(n,\mathbb{C})$. 

\end{cor}

\begin{proof}
From theorem \ref{sqf}, it suffices to compute the number of different isomorphism classes. When $n$ is squarefree, the cyclic group of order $n$ is the only abelian group of order $n$. Because of corollary \ref{ncompatible}, it suffices to compute the number of subgroups in $\mathbb{Z}/n\times \mathbb{Z}/n$. In order to do this, we associate to any subgroup $A$ of  $\mathbb{Z}/n\times \mathbb{Z}/n$ the $r$-uple $(A_{p_1},\dots,A_{p_r})$ of its $p_i$-Sylows. Since the isomorphism class of $A$ only depends on $(A_{p_1},\dots,A_{p_r})$, it suffices to compute the number of possible choices for the $p_i$-Sylows.

\bigskip

Since $A_{p_i}$ is clearly included in the $p_i$-Sylow of $\mathbb{Z}/n\times \mathbb{Z}/n$ which is isomorphic to $\mathbb{Z}/p_i\times\mathbb{Z}/p_i$, it follows that we have exactly three choices for $A_{p_i}$ : $\{0\}$, $\mathbb{Z}/p_i$ and $\mathbb{Z}/p_i\times\mathbb{Z}/p_i$. As a result we have $3^r$ different choices for the isomorphism class of $A$. \end{proof}

Now we would like to highlight a last  consequence that might be the most fruitful. Let $\mathcal{M}_n$ be the set of isometry classes of alternate modules which are $n$-subsymplectic. We   see $(\mathcal{M}_n,\leq)$ as an ordered set, where $\leq$ is the usual  relation of inclusion (up to isometry).  For any $n\geq 1$, we define a graph structure $G_n$, by taking for the set of vertices, the set $\mathcal{M}_n$ and any two classes of modules $M_1$ and $M_2$ are linked by an edge if $M_1\leq M_2$ and $|M_2|/|M_1|$ is a prime number or $M_2\leq M_1$ and $|M_1|/|M_2|$ is a prime number. 

\bigskip

Before giving some examples we give a notation for some alternate modules. For any finite abelian group $B$, the symplectic module $B\times B^*$ will be denoted $S(B)$. For any integer $k\geq 1$, the trivial module on $\mathbb{Z}/k$ will be denoted by $C_k$. For any triple $(e',e,d)$ of integers such that $e'$ divides $e$ and $e$ divides $d$ :

$$M_{e',e,d}:=\left(\mathbb{Z}/e\times \mathbb{Z}/d,\begin{pmatrix}0&-1/e'\\1/e'&0\end{pmatrix}\right)\text{.}$$ 

Furthermore, here, $\oplus$ will denote the orthogonal sum.

\begin{ex}\label{Mp1}

If $p$ is a prime number, there are exactly $3$ conjugacy classes of centralizers of irreducible subgroups in $PSL(p,\mathbb{C})$.

\end{ex}

\begin{proof}
Indeed, it suffices to compute the cardinal of $\mathcal{M}_{p}$ by theorem \ref{classif2}. In this case, the graph $G_p$ can easily seen to be like in figure \ref{pict1}.

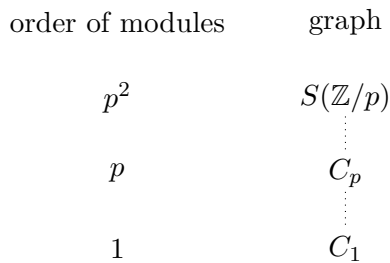
\begin{figure}[h]
\centering

\begin{tikzpicture}
\node at (-3,4) {order of modules};
\node at (0,4) {graph};
\node at (-3,3){$p^2$};
\node at (0,3){$S(\mathbb{Z}/p)$};
\node at (-3,2){$p$};
\node at (0,2) {$C_p$};
\node at (-3,1){$1$};
 \node at (0,1) {$C_1$};

 \draw[dotted] (0,2.75)--(0,2.25);
\draw[dotted] (0,1.75)--(0,1.25);

\end{tikzpicture}
\caption{The graph $G_p$ } \label{pict1}
\end{figure}

\end{proof}

\begin{ex}\label{Mp2}

If $p$ is a prime number, there are exactly $9$ conjugacy classes of centralizers of irreducible subgroups in $PSL(p^2,\mathbb{C})$. 

\end{ex}

\begin{proof}

Indeed, it suffices to compute the cardinal of $\mathcal{M}_{p^2}$ by theorem \ref{classif2}. In this case, the graph $G_{p^2}$ can   be computed to be  like in figure \ref{pict2}.

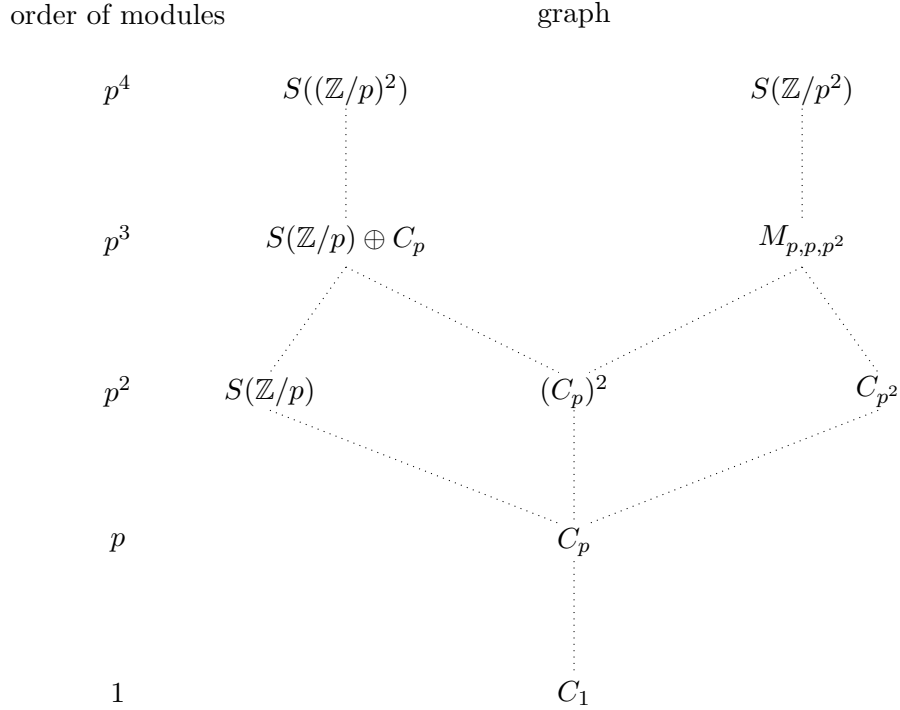
\begin{figure}
\centering

\begin{tikzpicture}
\node at (-6,5) {order of modules};
\node at (0,5) {graph};
\node at (-6,4){$p^4$};
\node at (-3,4){$S((\mathbb{Z}/p)^2)$};
\node at (3,4) {$S(\mathbb{Z}/p^2) $};
\node at (-6,2){$p^3$};
\node at (-3,2) {$S(\mathbb{Z}/p)\oplus C_p$};
\node at (3,2){$M_{p,p,p^2}$};
\node at (-6,0){$p^2$};
\node at (-4,0)   {$S(\mathbb{Z}/p)$};
\node at (0,0)   {$(C_p)^2 $};
\node at (4,0)   {$C_{p^2}$};
\node at (-6,-2)   {$p $};
\node at (0,-2)   {$C_p$};
\node at (-6,-4)   {$1$};
\node at (0,-4)   {$C_1$};

\draw[dotted] (-3,3.75)--(-3,2.25);
\draw[dotted]  (3,3.75)--(3,2.25);
\draw[dotted]  (-3,1.65)--(-4,0.25);
\draw[dotted]  (-3,1.65)--(-0.20,0.25);
\draw[dotted]  (3,1.65)--(4,0.25);
\draw[dotted]  (3,1.65)--(0.20,0.25);
 \draw[dotted] (-4,-0.25)--(-0.20,-1.75);
 \draw[dotted] (0,-0.25)--(0,-1.75);
 \draw[dotted] (4,-0.25)--(0.20,-1.75);
\draw[dotted] (0,-2.25)--(0,-3.75);

\end{tikzpicture}
\caption{The graph $G_{p^2}$ } \label{pict2}
\end{figure}\end{proof}

\begin{ex}\label{Mp3}

If $p$ is a prime number, there are exactly $24$ conjugacy classes of centralizers of irreducible subgroups in $PSL(p^3,\mathbb{C})$. 

\end{ex}

\begin{proof}

Indeed, it suffices to compute the cardinal of $\mathcal{M}_{p^3}$ by theorem \ref{classif2}. In this case, the graph $G_{p^3}$ can be computed like in figure \ref{pict3}.
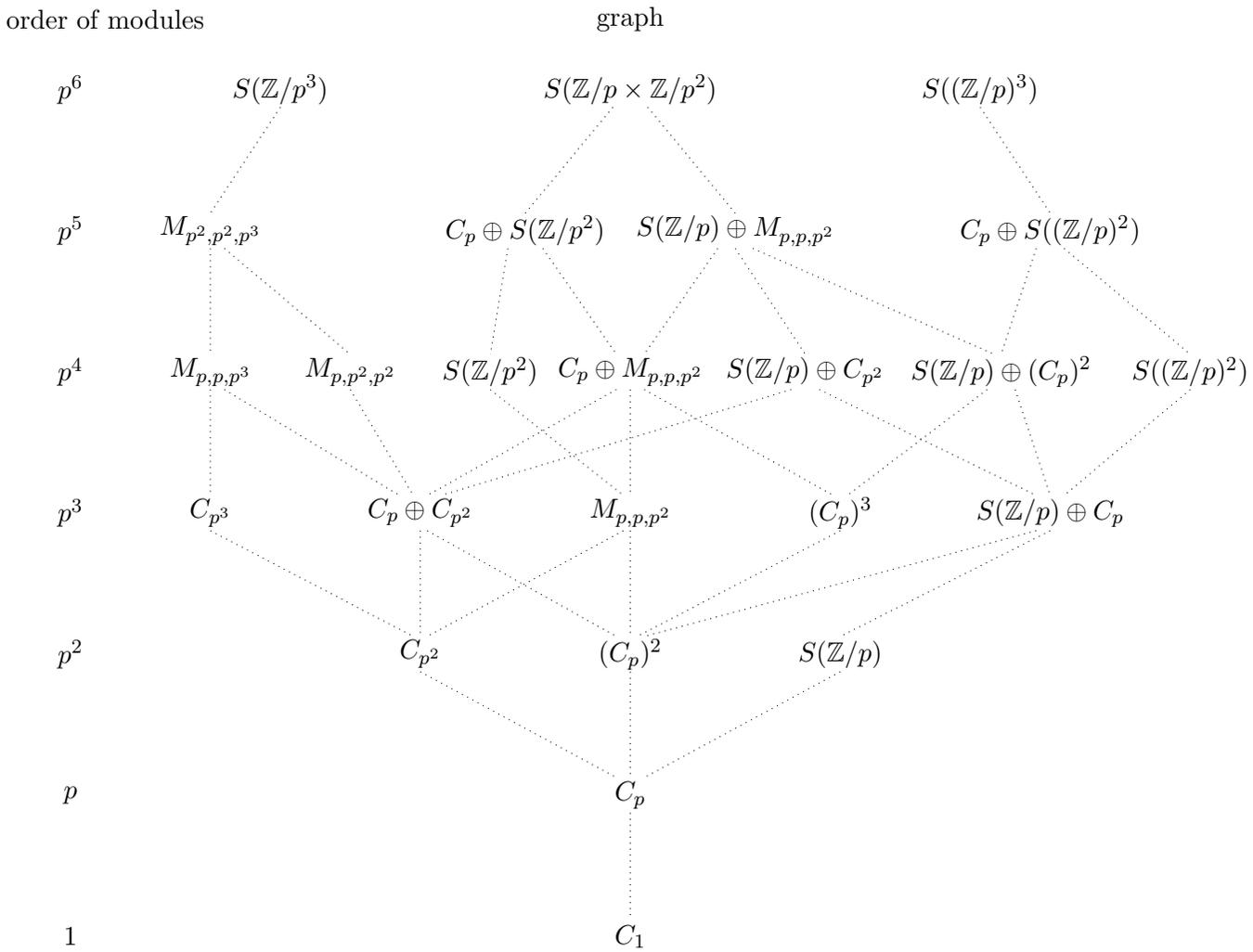
\begin{figure}[h!]
\centering

\begin{tikzpicture}
 \node at (-7,-7.5)[xscale=1,yscale=1,rotate=90]{ order of modules};
\node at (-6,-8)[xscale=1,yscale=1,rotate=90]{ $p^6$};
\node at (-4,-8)[xscale=1,yscale=1,rotate=90]{ $p^5$};
\node at (-2,-8)[xscale=1,yscale=1,rotate=90]{ $p^4$};
\node at (0,-8)[xscale=1,yscale=1,rotate=90]{ $p^3$};
\node at (2,-8)[xscale=1,yscale=1,rotate=90]{ $p^2$};
\node at (4,-8)[xscale=1,yscale=1,rotate=90]{ $p$};
\node at (6,-8)[xscale=1,yscale=1,rotate=90]{ $1$};

 \node at (-7,0)[xscale=1,yscale=1,rotate=90]{ graph};

\node at (-6,-5)[xscale=1,yscale=1,rotate=90]{ $S(\mathbb{Z}/p^3)$};
\node at (-6,0)[xscale=1,yscale=1,rotate=90]{ $S(\mathbb{Z}/p\times\mathbb{Z}/p^2)$};
\node at (-6,5)[xscale=1,yscale=1,rotate=90]{ $S((\mathbb{Z}/p)^3)$};

\node at (-4,-6)[xscale=1,yscale=1,rotate=90]{ $M_{p^2,p^2,p^3}$};
\node at (-4,-1.5)[xscale=1,yscale=1,rotate=90]{ $C_p\oplus S(\mathbb{Z}/p^2)$};
\node at (-4,1.5)[xscale=1,yscale=1,rotate=90]{ $S(\mathbb{Z}/p)\oplus M_{p,p,p^2}$};
\node at (-4,6)[xscale=1,yscale=1,rotate=90]{ $C_p\oplus S((\mathbb{Z}/p)^2)$};

\node at (-2,-6)[xscale=1,yscale=1,rotate=90]{ $M_{p,p,p^3}$};
\node at (-2,-4)[xscale=1,yscale=1,rotate=90]{ $M_{p,p^2,p^2}$};
\node at (-2,-2)[xscale=1,yscale=1,rotate=90]{ $S(\mathbb{Z}/p^2)$};
\node at (-2,0)[xscale=1,yscale=1,rotate=90]{ $C_p\oplus M_{p,p,p^2}$};
\node at (-2,2.5)[xscale=1,yscale=1,rotate=90]{ $S(\mathbb{Z}/p)\oplus C_{p^2}$};
\node at (-2,5.3)[xscale=1,yscale=1,rotate=90]{ $S(\mathbb{Z}/p)\oplus (C_p)^2$};
\node at (-2,8)[xscale=1,yscale=1,rotate=90]{ $S((\mathbb{Z}/p)^2)$};

\node at (0,-6)[xscale=1,yscale=1,rotate=90]{ $C_{p^3}$};
\node at (0,-3)[xscale=1,yscale=1,rotate=90]{ $C_p\oplus C_{p^2}$};
\node at (0,0)[xscale=1,yscale=1,rotate=90]{ $M_{p,p,p^2}$};
\node at (0,3)[xscale=1,yscale=1,rotate=90]{ $(C_p)^3$};
\node at (0,6)[xscale=1,yscale=1,rotate=90]{ $S(\mathbb{Z}/p)\oplus C_p$};

\node at (2,-3) [xscale=1,yscale=1,rotate=90]{ $C_{p^2}$};
\node at (2,0) [xscale=1,yscale=1,rotate=90]{ $(C_p)^2$};
\node at (2,3) [xscale=1,yscale=1,rotate=90]{ $S(\mathbb{Z}/p)$};

\node at (4,0)  [xscale=1,yscale=1,rotate=90]{ $C_p$};

\node at (6,0)  [xscale=1,yscale=1,rotate=90]{ $C_1$};

\draw[dotted] (-5.75,-5)--(-4.25,-6);
\draw[dotted] (-5.75,-0.25)--(-4.25,-1.5);
\draw[dotted] (-5.75,0.25)--(-4.25,1.5);
\draw[dotted] (-5.75,5)--(-4.25,6);

\draw[dotted] (-3.75,-6)--(-2.25,-6);
\draw[dotted] (-3.75,-5.80)--(-2.25,-4);
\draw[dotted] (-3.75,-1.75)--(-2.25,-2);
\draw[dotted] (-3.75,-1.25)--(-2.25,-0.20);
\draw[dotted] (-3.75,1.25)--(-2.25,0.20);
\draw[dotted] (-3.75,1.5)--(-2.25,2.5);
\draw[dotted] (-3.75,1.75)--(-2.25,5.2);
\draw[dotted] (-3.75,5.8)--(-2.25,5.3);
\draw[dotted] (-3.75,6.2)--(-2.25,8);

\draw[dotted] (-1.75,-6)--(-0.25,-6);
\draw[dotted] (-1.75,-5.80)--(-0.25,-3.3);
\draw[dotted] (-1.75,-4)--(-0.25,-3.1);
\draw[dotted] (-1.75,-2)--(-0.25,-0.1);
\draw[dotted] (-1.75,-0.2)--(-0.25,-2.9);
\draw[dotted] (-1.75,0)--(-0.25,0);
\draw[dotted] (-1.75,0.2)--(-0.25,2.9);
\draw[dotted] (-1.75,2.3)--(-0.25,-2.7);
\draw[dotted] (-1.75,2.7)--(-0.25,5.8);
\draw[dotted] (-1.75,5.1)--(-0.25,3.1);
\draw[dotted] (-1.75,5.5)--(-0.25,6);
\draw[dotted] (-1.75,8)--(-0.25,6.2);

\draw[dotted] (0.25,-6)--(1.75,-3.1);
\draw[dotted] (0.25,-3)--(1.75,-3);
\draw[dotted] (0.25,-2.9)--(1.75,-0.2);
\draw[dotted] (0.25,-0.1)--(1.75,-2.9);
\draw[dotted] (0.25,0)--(1.75,0);
\draw[dotted] (0.25,3)--(1.75,0.1);
\draw[dotted] (0.25,5.8)--(1.75,0.2);
\draw[dotted] (0.25,6)--(1.75,3);

\draw[dotted] (2.25,-3)--(3.75,-0.2);
\draw[dotted] (2.25,0)--(3.75,0);
\draw[dotted] (2.25,3)--(3.75,0.2);

\draw[dotted] (4.25,0)--(5.75,0);

\end{tikzpicture}
\caption{The graph $G_{p^3}$ } \label{pict3}
\end{figure}

\end{proof}

In \cite{Gue}, we defined for $G$ a complex semi-simple Lie group and $\Gamma$ a finitely generated group a \textit{bad} representation to be  an irreducible representation $\rho: \Gamma\rightarrow G$  with a non-trivial centralizer. We also defined the \textit{singular locus} of the character variety to be the set of conjugacy classes of bad representations. This set is denoted $\chi_{Sing}^i(\Gamma,G)$.  In the same paper, we used the theorem classifying centralizers of irreducible subgroups in $PSL(p,\mathbb{C})$ ($p$ is a prime number) to study, for any  finitely generated group $\Gamma$, the singular locus of the character variety $\chi_{Sing}^i(\Gamma,PSL(p,\mathbb{C}))$. 

\bigskip

Somehow, we hope that the same thing is possible when  the prime number is replaced by any integer $n$. Basically, the idea is the following. Let $M$ be the isometry class of an alternate module in $\mathcal{M}_n$. Because of theorem \ref{classif2}, there exists a centralizer $Z_M$ of  an irreducible subgroup in $PSL(n,\mathbb{C})$ whose associated alternate module is $M$. For any $M\in\mathcal{M}_n$, we define  a subset $\chi_M$ of the character variety :

$$\chi_M:=\left\lbrace PSL(n,\mathbb{C})\cdot \rho\left|\begin{array}{l}\rho\in Hom(\Gamma,Z_{PSL(n,\mathbb{C})}(Z_M))\\
\rho\text{ is irreducible}\end{array}\right.\right\rbrace\text{.} $$

Using the fact that $Z_M$ is well defined up to conjugation, $\chi_M$ does not depend on the conjugate of $Z_M$ we chose. For any $PSL(n,\mathbb{C})\cdot \rho$ in $\chi_M$, $\rho$ is centralized by some conjugate of $Z_M$. Furthermore, using again the unicity, up to conjugation of $Z_M$, if $\rho$ is an irreducible representation of $\Gamma$ into $PSL(n,\mathbb{C})$ such that the alternate module $M_{\rho}$ associated to its centralizer contains an isometric copy of $M$ then the orbit of $\rho$ is contained in $\chi_M$.  As a result, if we consider $(\mathcal{LS}_n,\leq)$ to be the set of $\chi_M$ for $M$ being an element of $\mathcal{M}_n$ and $\leq$ is the usual inclusion then $(\mathcal{LS}_n,\leq)$ is a partially ordered set which verifies  :

\begin{prop}\label{duality}

For $n\geq 2$ and $\Gamma$ a finitely generated group,  the application $\chi:\mathcal{M}_n\rightarrow \mathcal{LS}_n$ sending $M$ to $\chi_M$ is decreasing.  Furthermore :

$$\chi_{Sing}^i(\Gamma,PSL(n,\mathbb{C}))=\bigcup_{\substack{M\in\mathcal{M}_n\\ |M|>1}}\chi_M \text{.}$$

 Finally, if $M\in \mathcal{M}_n$ is of order $n^2$ then $\chi_M$ is finite. 

\end{prop}

\begin{proof}

Assume that $M\leq N$ with $M,N\in\mathcal{M}_n$. Then, we may choose $Z_M\leq Z_N$. Hence, we have $Z_{PSL(n,\mathbb{C})}(Z_N)\leq Z_{PSL(n,\mathbb{C})}(Z_M)$ which   leads to $\chi_N\leq \chi_M$. 

\bigskip

Let $\rho$ be a representation such that $PSL(n,\mathbb{C})\cdot\rho$ belongs to $\chi_M$ with $|M|>1$ then $\rho$ is irreducible by definition and $\rho$ is centralized by one conjugate of $Z_M$ which is not trivial. It follows that $Z(\rho)$ cannot be trivial and $\rho$ is bad representation. It follows that its conjugacy class belongs to the singular locus of the character variety.

\bigskip

Conversly, assume that $\rho$ is a bad representation, then define $M$ to be the isometry class of the alternate module associated to $Z(\rho)$. Applying theorem \ref{classif2}, $M\in \mathcal{M}_n$ and since $Z(\rho)$ is not trivial, $|M|>1$. By theorem \ref{classif1}, it follows that $Z(\rho)$ is conjugate to $Z_M$. Since $\rho$ is irreducible, we do get that its conjugacy class belongs to $\chi_M$ where $|M|>1$. Hence we get :

$$\chi_{Sing}^i(\Gamma,PSL(n,\mathbb{C}))=\bigcup_{\substack{M\in\mathcal{M}_n\\ |M|>1}}\chi_M \text{.}$$

For the last assertion, remark that if $M$ is of order $n^2$ then $Z_M$ is a full centralizer. By proposition \ref{full}, $Z_M$ is irreducible and self-centralizing. In particular $Z_{PSL(n,\mathbb{C})}(Z_M)=Z_M$. It follows that for any representation $\rho$ of $\Gamma$ such that its conjugacy class belongs to $\chi_M$, the image of $\rho$ must be equal to $Z_M$ (taking $H:=\pi_n^{-1}(\rho(\Gamma))$ in proposition \ref{full}).

\bigskip

Consider, in $\Gamma$, the set $X$ of normal subgroups $N$ such that $\Gamma/N$ is isomorphic to $Z_M$. Since $Z_M$ is abelian, $X$ is in bijection with the set $X'$ of subgroups $A$ in $\Gamma^{Ab}$ such that $\Gamma^{Ab}/A$ is isomorphic to $Z_M$. Since $\Gamma$ is finitely generated, $\Gamma^{Ab}$ is finitely generated. This implies that the set $X'$ is necessarily finite. For any representation $\rho$ whose conjugacy class belongs to  $\chi_M$, we get that $Ker(\rho)\in X$. 

\bigskip

If we denote $\chi_{M,N}$ to be the conjugacy classes of $\rho$ whose kernel is equal to $N$, then $\chi_{M,N}$ is exactly the set of conjugacy classes of  representations from $\Gamma/N$ onto $Z_M$. Since we are dealing with finite group representations, $\chi_{M,N}$ is necessarily finite. Since $\chi_M$ is a union of $\chi_{M,N}$ for $N\in X$ and $X$ is finite, it follows that $\chi_M$ is finite as well.  \end{proof}

This leads to a decomposition of the singular locus $\chi_{Sing}^i(\Gamma,PSL(n,\mathbb{C}))$ governed by the graph $G_n$ we defined earlier. Each vertex $M$ of $G_n$ leads to a stratum of the singular locus, and the minimal strata of the singular locus are simply a finite union of points, namely the $\chi_M$ where $M$ is a symplectic module of order $n^2$.

\section{Centralizers in quotients of $SL(n,\mathbb{C})$}\label{sec4}

In this section, we gather some generalizations of the results obtained for centralizers in $PSL(n,\mathbb{C})$ to centralizers of irreducible subgroups in $\pi_d(SL(n,\mathbb{C}))$ where $d$ divides $n$. Roughly speaking, we shall see that the results are almost the same.

\subsection{Abelianity, exponent and order}\label{sec4sub1}

First, we handle the abelianity of the $d$-centralizers.

\begin{prop}\label{dab}

Let $n\geq 1$, $d$ be a divisor of $n$ and $H$ be an irreducible subgroup of $SL(n,\mathbb{C})$, then $Z_d(H)$ is abelian.

\end{prop}

\begin{proof}
Let $u,v$ be two elements in $U_d(H)$ (we recall that $U_d(H):=\pi_d^{-1}(Z_d(H))$). Since we have $U_d(H)\leq U_n(H)$ and $Z_n(H)$ abelian (by proposition \ref{2p1}), it follows that $[u,v]$ belongs to $\langle \xi I_n\rangle$. In particular, lemma \ref{2l1} implies :

\begin{equation}
[u,v]^d=[u^d,v]\text{.}
\label{commutuv}
\end{equation}

On the other hand, since $u\in U_d(H)$, we also have that for all $h\in H$ $[u,h]$ is in $\langle \xi I_n\rangle$ of order dividing $d$. It follows (by lemma    \ref{2l1} again) that $[u^d,h]=I_n$ for all $h\in H$ i.e. $u^d$ centralizes $H$. Since $H$ is irreducible, Schur's lemma implies that $u^d\in \langle \xi I_n\rangle$. It follows, from equation \ref{commutuv}, that $[u,v]$ is of order dividing $d$. Since it is central,  we get, by definition, that $[u,v]\in \ker(\pi_d)$. As a result $Z_d(H)=\pi_d(U_d(H))$ is abelian.\end{proof}

The second proposition deals with the exponent of $Z_d(H)$. 

\begin{prop}\label{dexpo}
Let $n\geq 1$, $d$ be a divisor of $n$ and $H$ be an irreducible subgroup of $SL(n,\mathbb{C})$, then the exponent of $Z_d(H)$ divides $lcm(n/d,d)$. 

\end{prop}
\begin{proof}
Let $u$ be an element of $U_d(H)$ and $d_0$ be the order of $\pi_n(u)$. Let $h$ be in  $H$ such that $[h,u]$ is of order $d_0$. Since $u\in U_d(H)$, it follows that $[h,u]$ must also be of order dividing $d$. It follows that $d_0$ divides $d$. Using proposition \ref{3.diagonal} :

$$u\text{ is conjugate to }\lambda \begin{pmatrix}I_{\frac{n}{d_0}}&&&\\&\xi^{\frac{n}{d_0}}I_{\frac{n}{d_0}}&&\\&&\ddots&\\&&&\xi^{\frac{n}{d_0}(d_0-1)}I_{\frac{n}{d_0}}\end{pmatrix}\text{.} $$

$$\text{ where }\lambda\in \left\lbrace
\begin{array}{ll}
\langle \xi I_n\rangle  &  \mbox{if $d_0$ is odd or $d_0$  is even and $n/d_0$ even,}\\
\sqrt{\xi^{-\frac{n(d_0-1)}{d_0}}}\langle \xi I_n\rangle& \mbox{if $d_0$ is even and $n/d_0$ odd.}
\end{array}
\right.$$

Let $m$ be a common multiplicator of $n/d$ and $d$. Then :

$$u^m\text{ is conjugate to }\lambda^m \begin{pmatrix}I_{\frac{n}{d_0}}&&&\\&\xi^{\frac{n}{d_0}}I_{\frac{n}{d_0}}&&\\&&\ddots&\\&&&\xi^{\frac{n}{d_0}(d_0-1)}I_{\frac{n}{d_0}}\end{pmatrix}^m \text{.}$$

Since $d_0$ divides $d$ which divides $m$,  $u^m$ is conjugate to $\lambda^mI_n$. On one hand, if $d_0$ is odd, or $d_0$ is even and $n/d_0$ even, then $\lambda^m\in\langle \xi^{m}\rangle\leq \langle \xi^{n/d}\rangle$. In particular, $u^m$ is central of order dividing $d$, whence $u^m\in \ker(\pi_d)$ and $\pi_d(u)$ is of order dividing $m$. On the other hand, if $d_0$ is even and $n/d_0$ is odd then $\lambda=\sqrt{\xi^{-\frac{n(d_0-1)}{d_0}}}\xi^k$ where $0\leq k\leq n-1$. Remark that $2$ divides $m$ since it divides $d_0$. Hence :

$$\sqrt{\xi^{-\frac{n(d_0-1)}{d_0}}}^m=(\xi^{\frac{n}{d_0}})^{-\frac{m}{2}(d_0-1)}\text{ belongs to }\langle \xi^{n/d}\rangle \text{.}$$

Since $(\xi^k)^m$ is also in $ \langle \xi^{n/d}\rangle$,  $\lambda^m$ belongs to  $\langle \xi^{n/d}\rangle$. Whence $u^m$ is central of order dividing $d$ and $u^m\in \ker(\pi_d)$, so that $\pi_d(u)$ is of order dividing $m$.  In any case, all elements in $Z_d(H)$ are of order dividing $m$. This is, in particular, true for $m=lcm(n/d,d)$. \end{proof}

We finally handle the order :

\begin{prop}\label{dcard1}

Let $n\geq 1$, $d$ be a divisor of $n$ and $H$ be an irreducible subgroup of $SL(n,\mathbb{C})$ then $Z_d(H)$ is of order dividing $n^3/d$. 

\end{prop}

\begin{proof}

We know that $U_d(H)\leq U_n(H)$. Furthermore, $\pi_d(U_d(H))=Z_d(H)$. Hence :

\begin{align*}
|Z_d(H)|&=\frac{|U_d(H)|}{d}\\
&\text{divides } \frac{|U_n(H)|}{d}\text{ since  $U_d(H)\leq U_n(H)$}\\
&\text{divides }\frac{n|Z_n(H)|}{d}\text{ since $\pi_n(U_n(H))=Z_n(H)$}\\
&\text{divides }\frac{n^3}{d}\text{ by corollary  \ref{3.ncardinal}.}
\end{align*}

\end{proof}

In this proof, we used an inclusion which is, a priori, very weak : $U_d(H)\leq U_n(H)$.  The next example shows that the bound of proposition \ref{dcard1}  cannot be improved in general (computations are left to the reader).

\begin{ex}\label{3.excardinal2}

We study the case $n:=4$ and $d:=2$. There exists an exceptional isomorphism between $SL(4,\mathbb{C})$ and $Spin(6,\mathbb{C})$. This makes of  $SL(4,\mathbb{C})$ the universal cover of $SO(6,\mathbb{C})$. In particular,  $\pi_2(SL(4,\mathbb{C}))$ is identified to  $SO(6,\mathbb{C})$. We denote  $\pi_2:SL(4,\mathbb{C})\rightarrow SO(6,\mathbb{C})$ the induced projection. We denote $\overline{H}\leq SO(6,\mathbb{C})$, the subgroup of diagonal matrices in $SO(6,\mathbb{C})$. Then $\overline{H}$ is a finite group of order $32$, which is its own centralizer in  $SO(6,\mathbb{C})$. Let  $H:=\pi_2^{-1}(\overline{H})$ then $|Z_2(H)|=32=\frac{4^3}{2}=\frac{n^3}{d}$.

\end{ex}

 In order to improve the bound in some particular cases, we prove a lemma :

\begin{lemme}\label{dtor}

Let $n\geq 1$, $d$ be a divisor of $n$ and $H$ be an irreducible subgroup of $SL(n,\mathbb{C})$. Recall that we have $U_d(H)\leq U_n(H)$ and $U_n(H)/\langle\xi I_n\rangle$ is equal to $Z_n(H)$ then the quotient group $U_d(H)/\langle \xi I_n\rangle$, included in $Z_n(H)$, is equal to $(Z_n(H))_{(d)}$, the $d$-torsion of the abelian group $Z_n(H)$.

\end{lemme}

\begin{proof}
Let $u$ be in $U_d(H)$ then we  have that for all $h\in H$ $[u,h]$ is in $\langle \xi I_n\rangle$ of order dividing $d$. In particular $u^d$ (by lemma \ref{2l1}) commutes with $H$. It follows that $u^d$ is central by Schur's lemma, whence $u^d\in \langle \xi I_n\rangle$. As a result, $\pi_n(u)$ is of order dividing $d$. Finally this leads to $U_d(H)/\langle \xi I_n\rangle\leq (Z_n(H))_{(d)}$.

Let $u\in U_n(H)$ be such that $u\in \pi_n^{-1} ((Z_n(H))_{(d)})$. It follows that $u^d$ is central. Let $h\in H$, since $u\in U_n(H)$, we know that $[u,h]$ is central. By lemma \ref{2l1}, $[u,h]^d=[u^d,h]$ and since $u^d$ is central, $[u,h]$ is of order dividing $d$. By definition of $U_d(H)$, this implies that $u\in U_d(H)$. With the first inclusion, we have $U_d(H)/\langle \xi I_n\rangle= (Z_n(H))_{(d)}$.\end{proof}

This leads to the following proposition :

\begin{prop}\label{3.cardgen}

Let $n\geq 1$, $d$ dividing  $n$ and $H$ an irreducible subgroup of $SL(n,\mathbb{C})$. If  $gcd(d,n/d)=1$ then the order of  $Z_d(H)$ divides $nd$.

\end{prop}

\begin{proof}
 
 The order of  $Z_n(H)$ divides $n^2$ by corollary \ref{3.ncardinal}.  We decompose $n$ as a product of prime numbers $n=p_1^{a_1}\dots p_r^{a_r}$. Since $gcd(d,n/d)=1$ and the order of the $p_i$'s does not matter, we may as well as assume that  $d=p_1^{a_1}\dots p_s^{a_s}$ and $n/d=p_{s+1}^{a_{s+1}}\dots p_r^{a_r}$. 
 
 \bigskip
 
Since $Z_n(H)$ is abelian by proposition \ref{2p1} of order dividing $n^2$ (by corollary \ref{3.ncardinal}), there exists, for any $1\leq i\leq r$ a unique subgroup $S_i\leq Z_n(H)$ whose order is a power of  $p_i$ such that : $Z_n(H)=S_1\times\cdots\times S_r$. Clearly,  $(Z_n(H))_{(d)}=S_1\times\cdots\times S_s $ and since $|Z_n(H)|$ divides $n^2$, we deduce from this, that  $|(Z_n(H))_{(d)}|$ divides $d^2$. Lemma \ref{dtor} states that $U_d(H)=\pi_n^{-1}((Z_n(H))_{(d)})$. Hence:

\begin{align*}
|Z_d(H)|&=|\pi_d(U_d(H))|= |U_d(H)|/d\\
&= |\pi_n^{-1}((Z_n(H))_{(d)})|= n|(Z_n(H))_{(d)}|/d\\
\end{align*}

Since $|(Z_n(H))_{(d)}|$ divides $d^2$, it follows that $|Z_d(H)|$ divides $nd$. \end{proof}

We gather the results of this subsection in one single theorem :

\begin{thm}\label{casd}
Let $n\geq 1$, $d$ be a divisor of $n$ and $H$ be an irreducible subgroup of $SL(n,\mathbb{C})$ then $Z_d(H)$ is abelian of exponent dividing $lcm(n/d,d)$ and of order dividing $n^3/d$. If $gcd(n/d,d)=1$ then $Z_d(H)$ is of order dividing $nd$. 
\end{thm}

In the next subsection, we shall study the conjugacy classes of $d$-centralizer of irreducible subgroups in $SL(n,\mathbb{C})$ and briefly apply this to the study of the corresponding  character variety.

\subsection{Singularities in the associated character variety}\label{sec4sub2}

First, we remark that we can always project $Z_d(H)$ onto a subgroup $A_d(H)$ of $PSL(n,\mathbb{C})$. Since it is a subgroup of $Z_n(H)$ (by lemma \ref{dtor}) it follows, by corollary \ref{3.stabparsgrp}, that this group $A_d(H)$ is the centralizer of an irreducible subgroup of $PSL(n,\mathbb{C})$. In particular we can associate to $A_d(H)$ a structure of alternate module. Denote $M_d(H):=(A_d(H),\phi_{H,d})$ the alternate module that we get with this construction. Now we easily generalize theorem~\ref{classif1}. 

\begin{prop}\label{dclassif1}
Let $n\geq 1$, $d$ be a divisor of $n$ and $H_1$, $H_2$  be two irreducible subgroups of $SL(n,\mathbb{C})$ then $Z_d(H_1)$ and $Z_d(H_2)$ are conjugate if and only if $M_d(H_1)$ and $M_d(H_2)$ are isometric.
\end{prop}

\begin{proof}
First remark that $Z_d(H_1)$ is conjugate to $Z_d(H_2)$ if and only if their respective projections $A_d(H_1)$ and $A_d(H_2)$ are conjugate in $PSL(n,\mathbb{C})$. Since $A_d(H_1)$ and $A_d(H_2)$ are centralizer of irreducible subgroups in $PSL(n,\mathbb{C})$, $A_d(H_1)$ and $A_d(H_2)$ are conjugate if and only if $M_d(H_1)$ and $M_d(H_2)$ are isometric applying theorem \ref{classif1}. \end{proof}

Likewise we may generalize theorem \ref{classif2}.

\begin{prop}\label{dclassif2}
Let $n\geq 1$, $d$ be a divisor of $n$ and  $(A,\phi)$ be an alternate module then there exists an irreducible subgroup $H$ in $SL(n,\mathbb{C})$ such that $(A,\phi)$ is isometric to $M_d(H)$ if and only if $n_{A,\phi}$ (the order of Lagrangians in $(A,\phi)$) divides $n$ and $A$ is of exponent dividing $d$.
\end{prop}

\begin{proof}
Let $H$ be an irreducible subgroup in $SL(n,\mathbb{C})$ then, by definition $A_d(H)=\pi_n(U_d(H))$ is of exponent dividing $d$ by lemma \ref{dtor}. Whence $M_d(H)$ is an alternate module of $d$-torsion. Furthermore $M_d(H)$ is isometrically embedded in $(Z_n(H),\phi_H)$. Let $L$ be Lagrangian of $M_d(H)$. It follows that $L$ is isotropic in $(Z_n(H),\phi_H)$, in particular, $L\leq L_0$ where $L_0$ is a Lagrangian in $(Z_n(H),\phi_H)$. It follows that $n_{M_d(H)}$ divides $n_{Z_n(H),\phi_H}$ which divides $n$ (by theorem \ref{classif2}).

\bigskip

Conversly, let $(A,\phi)$ be an alternate module of exponent dividing $d$ and such that $n_{A,\phi}$ divides $n$. By theorem \ref{classif2}, there exists an irreducible subgroup $H$ in $SL(n,\mathbb{C})$ such that $(Z_n(H),\phi_H)$ is isometric to $(A,\phi)$. Since $Z_n(H)$ is of exponent dividing $d$, it follows that $\pi_n(U_d(H))=Z_n(H)$ by lemma \ref{dtor}. In particular $A_d(H)=Z_n(H)$ and hence $M_d(H)=(Z_n(H),\phi_H)$ by definition. Hence $M_d(H)$ is isometric to $(A,\phi)$. \end{proof}

Before stating a straightforward corollary we recall a few facts about the character variety. Let $\Gamma$ be a Fuchsian group and $G$ be a semi-simple complex Lie group. It is known that in such case, $\chi^i(\Gamma, G)$ admits a structure of orbifold (see \cite{Sik}). Furthermore the local isotropy group of $G\cdot \rho$ is, up to conjugation,  $Z_G(\rho)/Z(G)$.

\begin{cor}\label{isocarvar}
Let $n\geq 1$, $d$ be a divisor of $n$ and  $\rho$ be an irreducible representation of $\Gamma$ into $G:=\pi_d(SL(n,\mathbb{C}))$. Then the isotropy group of $G\cdot \rho$ in the orbifold  $\chi^i(\Gamma, G)$ is an abelian group $A$ of exponent dividing $d$ such that there exists an abelian group $B$ of order $n$ with $A$ included in $B\times B$.
\end{cor}
\begin{proof}
Let $H:=\pi_d^{-1}(\rho(\Gamma))$. Then $H$ is irreducible and $Z_G(\rho)=Z_d(H)$ by definition. By proposition \ref{dclassif2}, $M_d(H)$ is an alternate module of exponent dividing $d$ and with its Lagrangians of order dividing $n$. It follows that $A_d(H)=Z_n(H)$, the underlying group of $M_d(H)$, is of exponent dividing $d$. By corollary \ref{ncompatible}, there exists  an abelian group $B$ of order $n$  such that  $A_d(H)=Z_n(H)$ is included in $B\times B$. The result follows since $Z_G(\rho)/Z(G)=A_d(H)$. \end{proof}

In conclusion of this short section, we can say that studying the case of $PSL(n,\mathbb{C})$ is equivalent to studying the case of quotients of $SL(n,\mathbb{C})$ in general.


\end{document}